\documentclass[a4paper, 12pt]{article}
\pdfoutput=1

\usepackage[T1]{fontenc}
\usepackage[utf8]{inputenc}

\usepackage{amsfonts,amsmath,amssymb,mathtools,dsfont,microtype} 
\usepackage[dvipsnames]{xcolor} 
\usepackage[noBBpl]{mathpazo} 

\DeclareFontFamily{U} {cmr}{}

\DeclareFontShape{U}{cmr}{m}{n}{
	<-6> cmr5
	<6-7> cmr6
	<7-8> cmr7
	<8-9> cmr8
	<9-10> cmr9
	<10-12> cmr10
	<12-> cmr12}{}

\DeclareSymbolFont{Xcmr} {U} {cmr}{m}{n}
\DeclareMathSymbol{\Delta}{\mathord}{Xcmr}{'001}
\DeclareMathSymbol{\Upsilon}{\mathord}{Xcmr}{'007}
\DeclareMathSymbol{\Omega}{\mathord}{Xcmr}{'012}

\usepackage[labelsep = period, labelfont = bf, justification = centering]{caption}
\usepackage{float,graphicx,subcaption}
\DeclareCaptionSubType*[roman]{figure}

\usepackage{booktabs,makecell}
\usepackage{enumitem,mdwlist}
\setlist[itemize]{topsep=0ex,itemsep=0ex,parsep=0.4ex}
\setlist[enumerate]{topsep=0ex,itemsep=0ex,parsep=0.4ex}

\usepackage{pgf,tikz,ifthen,calc}
\usepackage{float, graphicx}
\usepackage{tkz-euclide}
\tkzSetUpPoint[fill=black, size = 3pt]
\tkzSetUpLine[color=black, line width=0.6pt]
\tkzSetUpCircle[color=black, line width=0.6pt]

\usepackage[left = 2.5cm, right = 2.5cm, top = 2.5cm, bottom = 2.5cm, headsep = 12pt, headheight = 15pt]{geometry}
\usepackage{parskip}

\usepackage[hyphens]{url} 
\usepackage[linktoc = all, hidelinks, colorlinks, unicode=true, pagebackref=true]{hyperref} 
\usepackage[capitalise, compress, nameinlink, noabbrev]{cleveref} 

\hypersetup{linkcolor={blue!70!black}, citecolor={green!70!black}, urlcolor={blue!70!black}}

\usepackage{amsthm,thmtools,thm-restate}
\declaretheorem[name = Theorem, numberwithin = section, style = plain]{thm}

\declaretheorem[name = Corollary, numberlike = thm, style = plain]{cor}
\declaretheorem[name = Conjecture, numberlike = thm, style = plain]{conj}
\declaretheorem[name = Definition, numberlike = thm, style = definition]{defi}
\declaretheorem[name = Lemma, numberlike = thm, style = plain]{lem}

\declaretheorem[name = Proposition, numberlike = thm, style = plain]{prop}
\declaretheorem[name = Problem, numberlike = thm, style = plain]{prob}
\declaretheorem[name = Remark, numberlike = thm, style = definition]{rmk}
\crefname{thm}{Theorem}{Theorems}
\crefname{lem}{Lemma}{Lemmas}
\crefname{obs}{Observation}{Observations}
\crefname{conj}{Conjecture}{Conjectures}
\crefname{claim}{Claim}{Claims}
\crefname{prob}{Problem}{Problems}
\crefname{prop}{Proposition}{Propositions}
\crefname{cor}{Corollary}{Corollaries}
\crefname{rmk}{Remark}{Remarks}

\newlist{property}{enumerate}{1}
\setlist[property]{label = \arabic*., ref = \arabic*}
\crefname{propertyi}{property}{properties}

\DeclareFontFamily{U}{matha}{\hyphenchar\font45}
\DeclareFontShape{U}{matha}{m}{n}{
	<5> <6> <7> <8> <9> <10> gen * matha
	<10.95> matha10 <12> <14.4> <17.28> <20.74> <24.88> matha12
}{}
\DeclareSymbolFont{matha}{U}{matha}{m}{n}

\DeclareMathSymbol{\specialuparrow}{\mathrel}{matha}{"D2}
\DeclareMathSymbol{\specialrightarrow}{\mathrel}{matha}{"D1}

\renewcommand*{\backref}[1]{}
\renewcommand*{\backrefalt}[4]{
	\ifcase #1 Not cited.%
	\or $\specialuparrow$#2%
	\else $\specialuparrow$#2%
	\fi%
}

\renewcommand{\epsilon}{\varepsilon}
\renewcommand{\ge}{\geqslant}
\renewcommand{\le}{\leqslant}
\renewcommand{\geq}{\geqslant}
\renewcommand{\leq}{\leqslant}

\renewcommand{\subset}{\subseteq}

\DeclarePairedDelimiter{\abs}{\lvert}{\rvert}

\DeclarePairedDelimiter{\set}{\{}{\}}

\newcommand{\defn}[1]{\textcolor{Maroon}{\emph{#1}}}

\newcommand*{\bE}{\mathbb{E}}

\newcommand*{\bP}{\mathbb{P}}

\newcommand*{\bZ}{\mathbb{Z}}

\newcommand*{\cO}{\mathcal{O}}

\title{The Zarankiewicz problem on tripartite graphs} 

\date{\today}

\begin{document}

\author{Francesco Di Braccio\footnotemark[1] \qquad Freddie Illingworth\footnotemark[2]}

\maketitle

\begin{abstract}
    In 1975, Bollob\'{a}s, Erd\H{o}s, and Szemer\'{e}di asked for the smallest $\tau$ such that an $n \times n \times n$ tripartite graph with minimum degree $n + \tau$ must contain $K_{t, t, t}$, conjecturing that $\tau = \cO(n^{1/2})$ for $t = 2$. We prove that $\tau = \cO(n^{1 - 1/t})$ which confirms their conjecture and is best possible assuming the widely believed conjecture that the Zarankiewicz number satisfies $z(n; t) = \Theta(n^{2 - 1/t})$. Our proof uses a density increment argument.
    
    We also construct an infinite family of extremal graphs that are pairwise far apart (requiring the change of $\Omega(n^2)$ edges to get between any two).
\end{abstract}

\renewcommand{\thefootnote}{\fnsymbol{footnote}} 


\footnotetext[1]{Department of Mathematics, London School of Economics, London, UK (\textsf{\href{mailto:francescodibraccio@outlook.com}{francescodibraccio\allowbreak@outlook\allowbreak.com}})}

\footnotetext[2]{Department of Mathematics, University College London, London, UK (\textsf{\href{mailto:f.illingworth@ucl.ac.uk}{f.illingworth@ucl.ac.uk}})}

\renewcommand{\thefootnote}{\arabic{footnote}} 

\section{Introduction}\label{sec:intro}

Tur\'{a}n-type problems are some of the oldest and most fundamental in Combinatorics. They ask how dense must a graph be in order to force the presence of a particular substructure. The first results were those of Mantel~\cite{Mantel07} and Tur\'{a}n~\cite{Turan41}, who determined the maximum number of edges in an $n$-vertex graph without a triangle or complete graph of size $r$, respectively. In 1946, Erd\H{o}s and Stone~\cite{ErdosStone46} generalised Tur\'an's theorem by showing that any $n$-vertex graph with at least $(1 - \frac{1}{r - 1} + o(1)) \binom{n}{2}$ edges contains a complete $r$-partite graph with $t$ vertices in each part, denoted \defn{$K_r(t)$}. The complete multipartite graphs are particularly important since the number of edges required to force the presence of these determines, up to lower order terms, the number of edges required to force the presence of any non-bipartite graph~\cite{ErdosSimonovits66}.

An old class of Tur\'{a}n-type problems that are still not well-understood asks for density conditions which guarantee a $K_r(t)$ inside a host graph that is itself $k$-partite (for some $k \geq r$). Questions of this type are often remarkably difficult, even for small values of $r$, $t$, and $k$. For instance, the notorious Zarankiewicz problem~\cite{Zarankiewicz51} from 1951 asks for the maximum number of edges, \defn{$z(n; t)$}, in a bipartite graph with parts of size $n$ that does not contain $K_{t, t} = K_2(t)$. The celebrated K\H{o}v\'{a}ri-S\'{o}s-Tur\'{a}n theorem~\cite{KST54} states that $z(n; t) = \cO(n^{2 - 1/t})$. Though this bound is believed to be tight for all $t \geq 2$, no matching lower bound is known for $t \geq 4$ despite receiving considerable attention -- we direct the reader to the comprehensive survey of Fur\"{e}di and Simonovits~\cite{FurediSimonovits13}. 

For $k$-partite host graphs with $k \geq 3$, even less is known. Interest in these problems can be traced back to the seminal 1975 paper of Bollob\'{a}s, Erd\H{o}s, and Szemer\'{e}di~\cite{BES75}. They initiated the extensive study of multipartite Tur\'{a}n-type problems~\cite{Jin92,SzaboTardos06,HaxellSzabo06,LTZ22} which has been central to Combinatorics with applications in graph arboricity~\cite{Alon88}, list colouring~\cite{Haxell01}, defective colouring~\cite{HST03}, and strong chromatic numbers~\cite{Haxell04}. 

One of their main questions, which has so far remained elusive, concerned the first case, $k = r = 3$. That is, the following tripartite version of the Zarankiewicz problem.

\begin{prob}[Bollob\'{a}s-Erd\H{o}s-Szemer\'{e}di]\label{prob:BES}
    For each $t$, what is the smallest $\tau = \tau(n)$ such that any $n \times n \times n$ tripartite graph with minimum degree at least $n + \tau$ contains $K_3(t)$\textup{?}
\end{prob}

In the above, an \defn{$n \times n \times n$ tripartite graph} is a tripartite graph whose parts have size $n$. The reason to centre the minimum degree at $n$ is that there are $n \times n \times n$ tripartite graphs with minimum degree $n$ which are bipartite, and thus contain no triangles (consider the graph $G$ with parts $V_1$, $V_2$, and $V_3$ that contains all edges between $V_1$ and $V_2$, and between $V_1$ and $V_3$). On the other hand, Graver, answering a question of Erd\H{o}s (cf.\ \cite{BES75}), showed that $\tau = 1$ suffices to guarantee the presence of a triangle, and so minimum degree $n$ is the threshold at which the tripartite nature of the problem appears.

Bollob\'{a}s, Erd\H{o}s, and Szemer\'{e}di~\cite{BES75} proved that\footnote{In their paper~\cite[Thm.~2.6]{BES75} they state $\tau = \cO(n^{1 - 1/t^2})$. However, as observed by Bhalkikar and Zhao~\cite{BZ23}, a factor of three is missing from their calculations.} $\tau = \cO(n^{1 - 1/(3t^2)})$ for $t \geq 2$. They highlighted the case $t = 2$, 
for which their result yields $\tau = \cO(n^{11/12})$, and they conjectured that the correct value is actually $\tau = \Theta(n^{1/2})$. 
Taking the graph $G$ defined above and adding an $n \times n$ bipartite $K_{2,2}$-free graph with minimum degree $\Omega(n^{1/2})$ between $V_2$ and $V_3$ provides an extremal example showing that $\Theta(n^{1/2})$ would be tight. 

\Cref{prob:BES} saw no progress until a paper of Bhalkikar and Zhao~\cite{BZ23} which gives an alternative proof of $\tau = \cO(n^{1 - 1/(3t^2)})$ and provides a second extremal example based on a blown-up $C_5$. They remarked that, for $t = 2$, ``$cn^{11/12}$ is a natural additive term for \cref{prob:BES} under typical approaches for extremal problems'' and proving $\Theta(n^{1/2})$ may be ``hard given the presence of many non-isomorphic constructions''. Recently, in simultaneous and independent work, Chen, He, Lo, Luo, Ma, and Zhao~\cite{CHLLMZ24} improved the bound to $\tau = \cO(n^{1 - 1/t(t + 1)})$ and introduced yet a new family of extremal examples.

The purpose of this paper is two-fold. The first is to prove Bollob\'{a}s, Erd\H{o}s, and Szemer\'{e}di's conjecture that $\tau = \Theta(n^{1/2})$ when $t = 2$ as well as improve the answer to \cref{prob:BES} to $\cO(n^{1 - 1/t})$. 

\begin{thm}\label{thm:main}
    For every positive integer $t$, there is a constant $K$ such that, if $G$ is an $n \times n \times n$ tripartite graph and $\delta(G) \geq n + K n^{1 - 1/t}$, then $G$ contains $K_{t, t, t}$.
\end{thm}

If $z(n; t) = \Omega(n^{2 - 1/t})$, as is widely believed, then this theorem is best possible up to the value of the constant $K$ and so \cref{thm:main} provides a complete answer to \cref{prob:BES}. 

The second purpose is to show that there are even more non-isomorphic constructions than previously thought. We exhibit, in \cref{sec:extremal}, an infinite family of extremal examples which are far from each other (requiring the change of $\Omega(n^2)$ edges to get between any two) and very different from those in~\cite{CHLLMZ24}. These constructions are based on the classical Andr\'{a}sfai graphs.

It should be remarked that recent work of Han and Zhao~\cite{HZ24} solves the variant of \cref{prob:BES} asking for the minimum number of edges guaranteeing a $K_{t, t, t}$. For the Erd\H{o}s-Stone theorem and Zarankiewicz problem the minimum degree and edge version are essentially equivalent.\footnote{The proof of the Erd\H{o}s-Stone theorem follows from the minimum degree version by discarding vertices of degree at most $(1 - 1/r + o(1))n$. If $z(n; t) = \Theta(n^{2 - 1/t})$, then for each $n$, by discarding vertices of low degree and randomly sampling within each vertex class, there are $n \times n$ bipartite graphs with minimum degree $\Theta(n^{1 - 1/t})$ that do not contain $K_{t, t}$.} However, for the tripartite version, the difference is more than cosmetic. The graph $G$, described above, shows that the answer to the edge problem is at least $2n^2$. Han and Zhao~\cite{HZ24} showed that the answer is $(2 + o(1)) n^2$ and all extremal examples have a similar structure. We will see, in \cref{sec:extremal}, that \cref{prob:BES} has many, very different extremal examples, including some with as few as $(3/2 + o(1))n^2$ edges. \Cref{prob:BES} lives in a sparser regime than the edge version.

The paper is organised as follows. In \cref{sec:extremal}, we discuss extremal examples and present our construction based on the Andr\'{a}sfai graphs. We introduce some counting tools for finding complete multipartite subgraphs in \cref{sec:counting_tools}. In \cref{sec:boosters}, we introduce the notion of a \emph{booster}, which is a useful type of edge that is guaranteed to have large codegree, and prove some structural lemmas about boosters in tripartite graphs. The proof of \cref{thm:main} is distributed over \cref{sec:main_thm}. In \cref{sec:initialconfig}, we argue that a balanced tripartite graph with large minimum degree either contains a dense subgraph of boosters or many edges with very large codegree. In \cref{sec:densebooster,sec:heavy}, we consider these two regimes separately and show that each of them forces the presence of a $K_{t, t, t}$. We conclude by discussing some open problems in \cref{sec:conclud}.

\section{Extremal examples}\label{sec:extremal}

In this section we describe an infinite family of extremal examples which show that \cref{thm:main} is best possible, up to the value of the constant $K$, assuming the widely-believed conjecture that the Zarankiewicz numbers satisfy $z(n; t) = \Omega(n^{2 - 1/t})$. 

\begin{thm}\label{thm:extremal}
    Suppose there are positive constants $c$ and $C$ such that $cn^{2 - 1/t} \leq z(n; t) \leq C n^{2 - 1/t}$ holds for all $n$. For every integer $k \geq \frac{2C}{3c}$ and for every $n$ that is a sufficiently large multiple of $3k$, there is a $K_{t, t, t}$-free $n \times n \times n$ tripartite graph $G_{k, n}$ with minimum degree $n + (c/4) \cdot n^{1 - 1/t}$. 
    
    Further, for $k \neq k'$, at least $\Omega(n^2)$ edges of $G_{k, n}$ need to be deleted to obtain a subgraph of $G_{k', n}$.
\end{thm}

We remind the reader that $z(n; t) = \cO(n^{2 - 1/t})$ is already known -- this is the classical K\H{o}v\'{a}ri-S\'{o}s-Tur\'{a}n theorem~\cite{KST54}. They believed~\cite[\S6]{KST54} that \ $z(n; t) = \Theta(n^{2 - 1/t})$ and so the hypothesis of \cref{thm:extremal} holds. Many researchers also conjecture this to be the case; see, for example, \cite[p.~257]{Furedi91}, \cite[\S6, Conj.~15]{Bollobas04}, and \cite[Conj.~2.24]{FurediSimonovits13}. The conjecture is known to hold for $t = 2$~\cite{KST54} and $t = 3$~\cite{Brown66}.

In our pursuit of \cref{thm:extremal}, we will start with the dense triangle-free $n \times n \times n$ tripartite graphs given by the following lemma.

\begin{lem}\label{lem:Andrasfai}
    For all positive integers $n, k, \sigma$ with $n = 3k \sigma$ there is a triangle-free $n \times n \times n$ tripartite graph $G'_{k, n} = (V_1, V_2, V_3; E_0)$ satisfying the following properties\textup{:} there are sets $U_1 \subset V_1$ and $U_2 \subset V_2$ such that
    \begin{property}[noitemsep, label = \textup{\arabic{*}.}, ref = \textup{\arabic{*}}]
        \item \label{property:1} $\abs{U_1} = \abs{U_2} = \sigma$,
        \item \label{property:2} every vertex $v \in U_1$ has $N(v) = V_2$,
        \item \label{property:3} every vertex $v \in U_2$ has $N(v) = V_1$,
        \item \label{property:4} every vertex $v \in V(G'_{k, n}) \setminus (U_1 \cup U_2)$ has degree $n + \sigma$.
    \end{property}
    Further, for $k \neq k'$, at least $\Omega(n^2)$ edges of $G'_{k, n}$ need to be deleted to obtain a subgraph of $G'_{k', n}$.
\end{lem}

We will prove \cref{lem:Andrasfai} in \cref{sec:Andrasfai}. The parameter $k$ will control the macroscopic structure of the graph $G$.
Note that every vertex of $G_0$ that is not in $U_1 \cup U_2$ has degree at least $n + \sigma$ and all vertices in $U_1 \cup U_2$ have degree $n$. We now add edges between $U_1$ and $V_3$ and between $U_2$ and $V_3$ according to the following lemma. 

\begin{lem}\label{lem:existenceKtt}
    If there are positive constants $c, C > 0$ such that $cn^{2 - 1/t} \leq z(n; t) \leq Cn^{2 - 1/t}$ holds for all $n$, then for any $\sigma \leq \frac{c}{2C} \cdot n$ there is a $\sigma \times n$ bipartite graph $H$ which does not contain $K_{t,t}$ and in which every vertex in the part of size $\sigma$ has degree at least $(c/4) \cdot n^{1 - 1/t}$.
\end{lem}

\begin{proof}
    Let $G$ be an $n \times n$ bipartite graph with $c n^{2 - 1/t}$ edges that does not contain $K_{t, t}$. Iteratively delete vertices of degree at most $(c/4) \cdot n^{1 - 1/t}$ from $G$ until no more remain. This deletes at most $(c/2) \cdot n^{2 - 1/t}$ edges in total and so at least $(c/2) \cdot n^{2 - 1/t}$ edges remain. Call the resulting graph $G_0 = (A_0, B_0; E_0)$ and let $n_0 \coloneqq \abs{A_0} \geq \abs{B_0}$. Since $G_0$ does not contain $K_{t, t}$, $(c/2) \cdot n^{2 - 1/t} \leq C n_0^{2 - 1/t}$ and so $n_0 \geq (\frac{c}{2C})^{t/(2t - 1)} \cdot n \geq \frac{c}{2C} \cdot n$. Note that $G_0$ has minimum degree at least $(c/4) \cdot n^{1 - 1/t}$. Deleting $\abs{A_0} - \sigma$ vertices from $A_0$ and adding $n - \abs{B_0}$ isolated vertices to $B_0$ gives the required bipartite graph.
\end{proof}

We may now prove \cref{thm:extremal} assuming \cref{lem:Andrasfai}.

\begin{proof}[Proof of \cref{thm:extremal}]
    Let $c$ and $C$ be positive constants such that $cn^{2 - 1/t} \leq z(n; t) \leq Cn^{2 - 1/t}$ and let $k \geq \frac{2C}{3c}$ be an integer. Let $\sigma$ be a large integer and $n = 3k \sigma$. Provided $\sigma$ is sufficiently large compared to $k$, we have $\sigma \geq (c/4) \cdot n^{1 - 1/t}$. Also $\sigma = \frac{n}{3k} \leq \frac{c}{2C} \cdot n$.

    Let $G'_{k, n}$ be the graph obtained from \cref{lem:Andrasfai} and let $H$ be the graph obtained from \cref{lem:existenceKtt}. We form $G_{k, n}$ by adding a copy of $H$ between $U_1$ and $V_3$ and a copy of $H$ between $U_2$ and $V_3$. Plainly $G$ is $n \times n \times n$ tripartite and has minimum degree at least $n + \min\set{\sigma, (c/4) \cdot n^{1 - 1/t}} = n + (c/4) \cdot n^{1 - 1/t}$.

    Let $k \neq k'$. Since the number of edges in copies of $H$ is at most $2z(n; t) = \cO(n^{2 - 1/t}) = o(n^2)$ and at least $\Omega(n^2)$ edges of $G'_{k, n}$ need to be deleted to obtain a subgraph of $G'_{k', n}$, the number of edges of $G_{k, n}$ that need to be deleted to obtain a subgraph of $G_{k', n}$ is $\Omega(n^2)$. 
    
    It remains to check that $G = G_{k, n}$ is $K_{t, t, t}$-free. Suppose there is a copy of $K_{t, t, t}$ in $G$: let its parts be $A_1 \subset V_1$, $A_2 \subset V_2$, and $A_3 \subset V_3$ where $\abs{A_1} = \abs{A_2} = \abs{A_3} = t$. The complete bipartite graph $G[A_1, A_3]$ is a copy of $K_{t, t}$. On the other hand, the graph $G[U_1, V_3]$ is isomorphic to $H$ and so is $K_{t, t}$-free. Thus, $A_1$ is not a subset of $U_1$. Similarly $A_2$ is not a subset of $U_2$. Let $v_1 \in A_1 \setminus U_1$, $v_2 \in A_2 \setminus U_2$, and $v_3 \in A_3$. Then $v_1 v_2 v_3$ is a triangle in $G$. However, none of these three vertices is in $U_1 \cup U_2$ and so $v_1 v_2 v_3$ is a triangle in $G'_{k, n}$. This contradicts the fact that $G'_{k, n}$ is triangle-free. Thus $G$ is indeed $K_{t, t, t}$-free.
\end{proof}

\subsection{\texorpdfstring{Andr\'{a}sfai}{Andrásfai} graphs}\label{sec:Andrasfai}

We will now prove \cref{lem:Andrasfai}. We start by discussing the \defn{Andr\'{a}sfai graphs}, first introduced by Andr\'{a}sfai~\cite{Andrasfai62}, which form an infinite family of 3-colourable triangle-free graphs with large minimum degree (compared to their number of vertices). The $k$th Andr\'{a}sfai graph, $\Gamma_k$, has vertex set $\bZ/(3k - 1)\bZ$ and the neighbourhood of each $i \in \bZ/(3k - 1)\bZ$ is $\set{i + k, i + k + 1, \dotsc, i + 2k - 1}$. This is a vertex-transitive graph (it is a Cayley graph). Following $\Gamma_1 = K_2$, the first few Andr\'{a}sfai graphs are depicted in \cref{fig:Andrasfai}. We remark that the extremal example of Bhalkikar and Zhao~\cite{BZ23} is based on the 5-cycle, which is $\Gamma_2$.

\begin{figure}[H]
    \centering
    \begin{subfigure}{0.2\textwidth}
        \centering
        \begin{tikzpicture}
            \foreach \i in {0,...,4}{
                \tkzDefPoint(72*\i+90:1){v_{\i}}
            }
            \tkzDrawPolygon[Blue](v_{0},v_{2},v_{4},v_{1},v_{3})
            \foreach \i in {0,...,4}{
                \tkzDrawPoint(v_{\i})            
            }
        \end{tikzpicture}
    \end{subfigure}
    \begin{subfigure}{0.2\textwidth}
        \centering
        \begin{tikzpicture}
            \foreach \i in {0,...,7}{
                \tkzDefPoint(45*\i+90:1){v_{\i}}
            }
            \tkzDrawPolygon[Blue](v_{0},v_{3},v_{6},v_{1},v_{4},v_{7},v_{2},v_{5})
            \tkzDrawSegments[Blue](v_{0},v_{4} v_{1},v_{5} v_{2},v_{6} v_{3},v_{7})
            \foreach \i in {0,...,7}{
                \tkzDrawPoint(v_{\i})         
            }
        \end{tikzpicture}
    \end{subfigure}
    \begin{subfigure}{0.2\textwidth}
        \centering
        \begin{tikzpicture}
            \foreach \i in {0,...,10}{
                \tkzDefPoint(32.7*\i+90:1){v_{\i}}
            }
            \tkzDrawPolygon[Blue](v_{0},v_{4},v_{8},v_{1},v_{5},v_{9},v_{2},v_{6},v_{10},v_{3},v_{7})
            \tkzDrawPolygon[Blue](v_{0},v_{5},v_{10},v_{4},v_{9},v_{3},v_{8},v_{2},v_{7},v_{1},v_{6})            
            \foreach \i in {0,...,10}{
                \tkzDrawPoint(v_{\i})            
            }
        \end{tikzpicture}
    \end{subfigure}
    \begin{subfigure}{0.2\textwidth}
        \centering
        \begin{tikzpicture}
            \foreach \i in {0,...,13}{
                \tkzDefPoint(25.7*\i+90:1){v_{\i}}
            }
            \tkzDrawPolygon[Blue](v_{0},v_{5},v_{10},v_{1},v_{6},v_{11},v_{2},v_{7},v_{12},v_{3},v_{8},v_{13},v_{4},v_{9})
            \tkzDrawPolygon[Blue](v_{0},v_{6},v_{12},v_{4},v_{10},v_{2},v_{8})
            \tkzDrawPolygon[Blue](v_{1},v_{7},v_{13},v_{5},v_{11},v_{3},v_{9})
            \tkzDrawSegments[Blue](v_{0},v_{7} v_{1},v_{8} v_{2},v_{9} v_{3},v_{10} v_{4},v_{11} v_{5},v_{12} v_{6},v_{13})
            \foreach \i in {0,...,13}{
                \tkzDrawPoint(v_{\i})            
            }
        \end{tikzpicture}
    \end{subfigure}
    \caption{The Andr\'{a}sfai graphs $\Gamma_2$, $\Gamma_3$, $\Gamma_4$, and $\Gamma_5$.}\label{fig:Andrasfai}
\end{figure}

The graph $\Gamma_{k + 1}$ has the following properties.
\begin{itemize}
    \item It is edge-maximal triangle-free: it is triangle-free and the addition of any edge results in a triangle.
    \item It is $(k + 1)$-regular on $3k + 2$ vertices.
    \item It is 3-partite with one possible partition being $\set{0, 1, \dotsc, k}$, $\set{k + 1, k + 2, \dotsc, 2 k + 1}$, and $\set{2k + 2, 2k + 3, \dotsc, 3k + 1}$.
\end{itemize}
Notice that the first two parts of the partition given above contain $k + 1$ vertices while the final part contains only $k$. 

In order to obtain a balanced tripartite graph we will need to weight the vertices of $\Gamma_{k + 1}$ so that each part has an equal weight. To be precise, given a graph $G$ and a \defn{weighting} $\omega \colon V(G) \to \bZ^+$, the \defn{weighted graph} $(G, \omega)$ is the graph obtained by replacing vertex $v$ of $G$ by an independent set $I_v$ of size $\omega(v)$ and each edge $uv$ by a complete bipartite graph between classes $I_u$ and $I_v$. Crucially, if $G$ is triangle-free, then the graph $(G, \omega)$ will be too. We will need the following lemma about weighted graph for our claim about the $G'_{k, n}$ being $\Omega(n^2)$ edges apart.

\begin{lem}\label{lem:homomorphism}
    Let $G$ be an edge-maximal triangle-free graph in which all neighbourhoods are distinct and $H$ be a triangle-free graph. Suppose that $G$ is a subgraph of some weighted graph $(H, \omega)$. Then $G$ is a subgraph of $H$.
\end{lem}

\begin{proof}
    A \defn{homomorphism} from a graph $G_1$ to a graph $G_2$ is a function $\phi \colon G_1 \to G_2$ such that, for every edge $uv \in G_1$, we have $\phi(u) \phi(v) \in G_2$ (that is, $\phi$ is edge-preserving). If $G$ is a subgraph of $(H, \omega)$, then there is a homomorphism from $G$ to $(H, \omega)$ (the function mapping $G$ to its embedding in $(H, \omega)$). Projecting this function down from $(H, \omega)$ to $H$ gives a homomorphism from $G$ to $H$ -- denote this homomorphism by $\phi$. 
    
    First suppose that $\phi$ is not injective. Then there are vertices $u, v \in V(G)$ such that $\phi(u) = \phi(v)$. Since all neighbourhoods of $G$ are distinct, we may and will assume there is a vertex $w$ of $G$ that is adjacent to $u$ and not to $v$. Let $G' = G + vw$. Now $\phi(v) \phi(w) = \phi(u) \phi(w)$ and $u$ is adjacent to $w$ so, in fact, $\phi$ is a homomorphism from $G'$ to $H$. But $G$ is an edge-maximal triangle-free graph and so $G'$ contains a triangle. This triangle will map, under $\phi$, to a triangle in $H$. This contradicts the triangle-freeness of $H$. Thus $\phi$ is injective. But then $\phi$ is an embedding of $G$ into $H$ and so $G$ is a subgraph of $H$.
\end{proof}

\Cref{fig:weighting} displays $\Gamma_{k + 1}$ together with a weighting $\omega$ (the number by each vertex $v$ is the weight $\omega(v)$). To orient the reader, the rightmost vertex (which has weight 1) corresponds to the vertex 0 in $\Gamma_{k + 1}$ and the other vertices proceed anticlockwise. We have drawn the edges incident to the top left vertex (which has weight 2 and corresponds to vertex $k + 1$ in $\Gamma_{k + 1}$) as an aid.

\begin{figure}[ht!]
    \centering
    \begin{tikzpicture}
        \foreach \i in {0,1,2}{
            \tkzDefPoint(120*\i-90:2){v_{\i}}
            \tkzDrawPoint(v_{\i})
            \tkzDrawEllipse[white, fill = gray!20](v_{\i},0.4,2.4,120*\i-90)
            \foreach \j in {-4,...,4}{
                \tkzDefShiftPoint[v_{\i}](120*\i:0.5*\j){v_{\i,\j}}
            }
            \foreach \j in {-4,-3,-2,2,3,4}{
                \tkzDrawPoint(v_{\i,\j})
            }
            \tkzLabelPoint[shift={(120*\i+90:0.2)},rotate=120*\i](v_{\i}){$\dotsb$}
        }
        \foreach \j in {-4,-3,-2,2,3,4}{
            \tkzLabelPoint[below](v_{0,\j}){3}
        }

        \tkzLabelPoint[above left](v_{2,-4}){2}
        \tkzLabelPoint[above left](v_{2,4}){1}
        \tkzLabelPoint[above right](v_{1,4}){2}
        \tkzLabelPoint[above right](v_{1,-4}){1}

        \foreach \j in {-3,-2,2,3}{
            \tkzLabelPoint[above left](v_{2,\j}){3}
            \tkzLabelPoint[above right](v_{1,\j}){3}
        }

        \foreach \i in {0,1,2}{
            \foreach \j in {-4,0,4}{
                \tkzDefShiftPoint[v_{\i,\j}](120*\i-90:0.6){u_{\i,\j}}
            }
            \draw [decorate,decoration={brace,mirror, amplitude=10pt},xshift=0pt,yshift=-100pt] (u_{\i,-4}) -- (u_{\i,4});
        }

        \tkzLabelPoint[below=8pt](u_{0,0}){$k$}
        \tkzLabelPoint[above right = 8pt](u_{1,0}){$k + 1$}
        \tkzLabelPoint[above left = 8pt](u_{2,0}){$k + 1$}

        \tkzDrawSegments(v_{1,-4},v_{2,-4} v_{0,-4},v_{2,-4} v_{0,-3},v_{2,-4} v_{0,-2},v_{2,-4} v_{0,2},v_{2,-4} v_{0,3},v_{2,-4} v_{0,4},v_{2,-4})
    \end{tikzpicture}
    \caption{The graph $(\Gamma_{k + 1}, \omega)$ -- the grey ellipses are the three parts.}\label{fig:weighting}
\end{figure}

With the weighting $\omega$ shown in \cref{fig:weighting} the total number of vertices in (i.e.\ the total weight of) each part is $3k$. Furthermore, every vertex has degree exactly $3k + 1$ except for the two vertices of weight 1: the neighbourhood of the left vertex $v_{\text{left}}$ of weight 1 is the top right part and the neighbourhood of the right vertex $v_{\text{right}}$ of weight 1 is the top left part. We are ready to prove \cref{lem:Andrasfai}.

\begin{proof}[Proof of \cref{lem:Andrasfai}]
    Let $\sigma$ be a positive integer. Consider the tripartite graph $G'_{k, n} \coloneqq (\Gamma_{k + 1}, \omega_\sigma)$ where $\omega_\sigma(v) = \sigma \cdot \omega(v)$. Each part contains $n \coloneqq 3k\sigma$ vertices and, since $\Gamma_{k + 1}$ is triangle-free, $G'_{k, n}$ is triangle-free. Each vertex has degree exactly $(3k + 1) \sigma = n + \sigma$ except for the vertices in the independent sets $I_{v_{\text{left}}}$ and $I_{v_{\text{right}}}$ corresponding to $v_{\text{left}}$ and $v_{\text{right}}$. Each of these independent sets has size $\sigma$ (since $\omega_{\sigma}(v_{\text{left}}) = \omega_{\sigma}(v_{\text{right}}) = \sigma$) and the neighbourhood of every vertex in $I_{v_{\text{left}}}$ is the top right part and the neighbourhood of every vertex in $I_{v_{\text{right}}}$ is the top left part. Thus the graph $G'_{k, n} \coloneqq (\Gamma_{k + 1}, \omega_\sigma)$ satisfies \cref{property:1,property:2,property:3,property:4} with $U_1 = I_{v_{\text{left}}}$ and $U_2 = I_{v_{\text{right}}}$.

    We are left to show that, for $k \neq k'$, at least $\Omega(n^2)$ edges of $G'_{k, n}$ need to be deleted to obtain a subgraph of $G'_{k', n}$. First suppose that $k < k'$. Now, the number of edges of $G_{k, n}$ is
    \begin{equation*}
        \tfrac{1}{2}[(3n - 2\sigma)(n + \sigma) + 2\sigma n] = \tfrac{1}{2}[3n^2 + 3\sigma n - 2 \sigma^2] = \tfrac{n^2}{2} [3 + \tfrac{1}{k} - \tfrac{2}{9 k^2}],
    \end{equation*}
    and so $e(G'_{k, n}) - e(G'_{k', n}) = \Omega(n^2)$. Next suppose that $k > k'$. We first show that $G'_{k', n}$ does not contain $\Gamma_{k + 1}$. The graph $\Gamma_{k + 1}$ is an edge-maximal triangle-free graph in which all neighbourhoods are distinct. The graph $\Gamma_{k' + 1}$ is triangle-free and $G'_{k', n}$ is a weighted version of $\Gamma_{k' + 1}$. Since $k > k'$, $\Gamma_{k + 1}$ is not a subgraph of $\Gamma_{k' + 1}$ and so \cref{lem:homomorphism} implies that $G'_{k', n}$ does not contain $\Gamma_{k + 1}$.
    
    On the other hand, the graph $G'_{k, n}$ is a weighted version of $\Gamma_{k + 1}$. A \defn{canonical copy} of $\Gamma_{k + 1}$ in $G'_{k, n}$ is a copy of $\Gamma_{k + 1}$ in which each vertex is in the corresponding part of $G'_{k, n}$. The number of such canonical copies is $(3 \sigma)^k \cdot (3 \sigma)^{k - 1} (2 \sigma) \sigma \cdot (3\sigma)^{k - 1} (2 \sigma) \sigma = 3^{3k - 2} \cdot 2^2 \cdot \sigma^{3k + 2}$. Each edge of $G'_{k, n}$ is in at most $(3 \sigma)^k \cdot (3 \sigma)^{k - 1} (2 \sigma) \cdot (3 \sigma)^{k - 1} (2 \sigma) = 3^{3k - 2} \cdot 2^2 \cdot \sigma^{3k}$ canonical copies of $\Gamma_{k + 1}$. Thus, to remove all copies of $\Gamma_{k + 1}$ from $G'_{k, n}$ (which is necessary to obtain a subgraph of $G'_{k', n}$) requires the deletion of at least $\sigma^2 = \Omega(n^2)$ edges, as required.
\end{proof}

\section{Counting tools}\label{sec:counting_tools}

In this section we collect together some counting lemmas that guarantee the presence of $K_{t, t}$ and $K_{t, t, t}$. Throughout the paper we use the following standard notation. \defn{$N_G(v_1, \dots, v_t)$} denotes the common neighbourhood of a collection of vertices $v_1, \dots, v_t$ in graph $G$. For a pair of vertices $a$ and $b$, \defn{$\deg_G(a, b)$} $\coloneqq \abs{N_G(a, b)}$ denotes the number of common neighbours of $a$ and $b$. For a pair of disjoint sets $A, B \subset V(G)$, \defn{$e_G(A, B)$} $\coloneqq \abs{E(G[A, B])}$ denotes the number of edges of $G$ with one end-vertex in $A$ and the other in $B$. Finally, for a vertex $a$ and vertex-set $B$, \defn{$e_G(a, B)$} $\coloneqq e_G(\set{a}, B)$ denotes the number of neighbours of $a$ in $B$. We omit the subscript if $G$ is clear from context.

We will frequently use the unbalanced version of the K\H{o}v\'{a}ri-S\'{o}s-Tur\'{a}n theorem~\cite{KST54}.

\begin{thm}\label{thm:KST}
    For every positive integer $t$ there is a constant $K$ such that any $m \times n$ bipartite graph with at least $K(mn^{1 - 1/t} + n)$ edges contains $K_{t, t}$. In particular, if $m \geq n^{1/t}$, then $2Kmn^{1-1/t}$ edges suffice to guarantee a $K_{t, t}$. \end{thm}

The following averaging argument converts many copies of $K_{1, 1, t}$ into a copy of $K_{t, t, t}$. Here and throughout the paper, when we write ``there are at least $N$ $K_{1, 1, t}$s in $A \times B \times C$'' we mean there are at least $N$ copies of $K_{1, 1, t}$ with one vertex in $A$, one in $B$, and $t$ in $C$.

\begin{lem}\label{lem:K11t}
    For every positive integer $t$ there is a constant $K \geq 1$ such that if $G = (A, B, C; E)$ is a tripartite graph with $\abs{A} \geq \abs{B}^{1/t}$ and there are at least $K \cdot \abs{A} \cdot \abs{B}^{1 - 1/t} \cdot \abs{C}^t$ $K_{1, 1, t}$s in $A \times B \times C$, then $G$ contains $K_{t, t, t}$.
\end{lem}

\begin{proof}
    By averaging over $t$-tuples of distinct vertices of $C$ there exist $c_1, \dotsc, c_t \in C$ such that $G[N(c_1, \dotsc, c_t)]$ contains at least $K \cdot t! \cdot \abs{A} \cdot \abs{B}^{1 - 1/t}$ edges. By \cref{thm:KST}, provided $K$ is sufficiently large in terms of $t$, the edges of $G[N(c_1, \dotsc, c_t)]$ contain a copy of $K_{t, t}$. Together with $c_1, \dotsc, c_t$ this gives a copy of $K_{t, t, t}$ in $G$.
\end{proof}

We will frequently apply \cref{lem:K11t} in the case where every edge between $A$ and $B$ has a linear-sized common neighbourhood in $C$. We package this into the following corollary.

\begin{cor}\label{cor:K11tdense}
    For every positive integer $t$ and constant $\lambda > 0$ there is a constant $K$ such that the following holds. Let $G = (A, B, C; E)$ be a tripartite graph where
    \begin{itemize}
        \item $\deg(a, b) \geq \max\set{\lambda \abs{C}, t}$ for every edge $ab \in E(G[A, B])$,
        \item $\abs{A} \geq \abs{B}^{1/t}$
        \item $e(G[A, B]) \geq K \cdot \abs{A} \cdot \abs{B}^{1 - 1/t}$.
    \end{itemize}
    Then $G$ contains a $K_{t, t, t}$.
\end{cor}

\begin{proof}
    The number of $K_{1, 1, t}$s in $A \times B \times C$ is
    \begin{equation*}
        \sum_{ab \in E(G[A, B])} \binom{\deg(a, b)}{t} \geq e(G[A, B]) \cdot \binom{\max\set{\lambda \abs{C}, t}}{t} \geq e(G[A, B]) \cdot (\lambda/t)^t \cdot \abs{C}^t,
    \end{equation*}
    where the last inequality is immediate if $t \geq \lambda \abs{C}$ and follows from $\binom{x}{t} \geq (x/t)^t$ for all $x \geq t$ otherwise. The result follows from \cref{lem:K11t}.
\end{proof}

We finish with a version of the dependent random choice lemma.

\begin{lem}\label{lem:coneigh_density}
    Let $G = (A, B; E)$ be a bipartite graph with edge density $\lambda$. If $\lambda \abs{A} \geq t$, then there are distinct vertices $a_1, \dotsc, a_t \in A$ and $B' \subset N(a_1, \dotsc, a_t)$ such that 
    \begin{itemize}[noitemsep]
        \item $\abs{B'} \geq \frac{1}{2} (\lambda/e)^t \abs{B}$ and
        \item every vertex in $B'$ has at least $\frac{1}{2} \lambda \abs{A}$ neighbours in $A$.
    \end{itemize}
\end{lem}

\begin{proof}
    Call a vertex $b \in B$ \defn{bad} if it has fewer than $\frac{1}{2} \lambda \abs{A}$ neighbours in $A$. Pick $t$ distinct vertices $a_1, \dotsc, a_t \in A$ uniformly at random. Let $X$ be the size of $N(a_1, \dotsc, a_t)$ and $Y$ be the number of bad vertices in $N(a_1, \dotsc, a_t)$. Using the convexity of $x \mapsto \binom{x}{t}$,
    \begin{align*}
        \bE(X) & = \sum_{b \in B} \bP(a_1, \dotsc, a_t \in N(b))
        = \sum_{b \in B} \binom{\deg(b)}{t} \binom{\abs{A}}{t}^{-1} \\
        & \geq \binom{\abs{A}}{t}^{-1} \abs{B} \cdot \binom{\abs{B}^{-1} \sum_{b \in B} \deg(b)}{t} = \binom{\abs{A}}{t}^{-1} \abs{B} \binom{\lambda \abs{A}}{t}.
    \end{align*}
    On the other hand,
    \begin{equation*}
        \bE(Y) = \sum_{\text{bad } b} \binom{\deg(b)}{t} \binom{\abs{A}}{t}^{-1} \leq \abs{B} \binom{\frac{1}{2} \lambda \abs{A}}{t} \binom{\abs{A}}{t}^{-1},
    \end{equation*}
    and so, since $\binom{\frac{1}{2} \lambda \abs{A}}{t} \leq \frac{1}{2} \binom{\lambda \abs{A}}{t}$,
    \begin{equation*}
        \bE(X - Y) \geq \abs{B} \binom{\abs{A}}{t}^{-1} \biggl[\binom{\lambda \abs{A}}{t} - \binom{\frac{1}{2} \lambda \abs{A}}{t}\biggr] \geq \frac{\abs{B}}{2} \binom{\lambda \abs{A}}{t} \binom{\abs{A}}{t}^{-1} \geq \frac{\abs{B}}{2} \biggl(\frac{\lambda}{e}\biggr)^t,
    \end{equation*}
    where the last inequality used the bounds $(x/t)^t \leq \binom{x}{t} \leq (ex/t)^t$ for all $x \geq t$.
    Pick distinct $a_1, \dotsc, a_t \in A$ for which $X - Y \geq \frac{1}{2} (\lambda/e)^t \abs{B}$. Then take $B'$ to be the set of good vertices in $N(a_1, \dotsc, a_t)$. This has size $X - Y$, as required.
\end{proof}

\section{Degrees, boosters, and some structural lemmas}\label{sec:boosters}

We now come to the proof of \cref{thm:main}. Recall that we start with an $n \times n \times n$ tripartite graph $G$ which has minimum degree at least $n + \tau$ where $\tau = K n^{1 - 1/t}$ for some large constant $K$. We label the three parts of $G$ as $V_1, V_2, V_3$. This invokes a natural direction of edges in $G$: for $u \in V_i$, $v \in V_{i + 1}$ (such indices will always be modulo 3) we say that $uv \in E(G)$ is a \defn{forward edge} while $vu \in E(G)$ is a \defn{backward edge}.

The \defn{forward neighbourhood} of a vertex $v$, denoted $N^+(v)$, is the set of $u$ such that $vu$ is a forward edge. Note that if $u \in V_i$, then $N^+(v) \subset V_{i + 1}$. The size of $N^+(v)$ is the \defn{forward degree} of $v$ and is denoted $\deg^+(v)$. The \defn{backward neighbourhood} and \defn{backward degree} are defined similarly. Note that $\deg(v) = \deg^+(v) + \deg^-(v)$ for each vertex $v$.


We now define certain edges that are particularly important in the proof as they are useful for creating $K_{t, t, t}$s.

\begin{defi}[boosters]\label{def:booster}
    Let $e$ be an edge. Let the end-vertices of $e$ be $u$ and $v$ where $uv$ is in the forward direction.
    \begin{itemize}[noitemsep]
        \item $e$ is a \defn{booster} if $\deg^+(v) \geq \deg^+(u)$,
        \item $e$ is an \defn{$r$-booster} if $\deg^+(v) \geq \deg^+(u) + r$.
    \end{itemize}
\end{defi}

Note that boosters are exactly 0-boosters. Boosters are very useful as they give a lot of information about the neighbourhoods of their end-points.

\begin{lem}\label{lem:booster}
    Let $G$ be an $n \times n \times n$ tripartite graph with $\delta(G) \geq n + \tau$. Let $uv$ be an $r$-booster and $V_i$ be the part containing neither $u$ nor $v$. Then
    \begin{itemize}
        \item $\deg(u, v) \geq \tau + r$,
        \item $\abs{V_i \setminus (N(u) \cup N(v))} \leq \deg(u, v) - \tau - r$.
    \end{itemize}
\end{lem}

\begin{proof}
    Without loss of generality $u \in V_{i + 1}$ and $v \in V_{i + 2}$ (indices modulo 3) and so $\deg^+(v) \geq \deg^+(u) + r$.
    The inclusion-exclusion principle gives
    \begin{align*}
        \deg(u, v) & = \abs{N(u, v) \cap V_i} \geq \abs{N(u) \cap V_i} + \abs{N(v) \cap V_i} - n \\
        & = \deg^+(v) + \deg^-(u) - n \\
        & = (\deg^+(v) - \deg^+(u)) + (\deg(u) - n) \\
        & \geq r + \tau,
    \end{align*}
    and
    \begin{align*}
        \abs{V_i \setminus (N(u) \cup N(v))} & = n - \abs{V_i \cap (N(u) \cup N(v))} \\
        & = n - \deg^-(u) - \deg^+(v) + \deg(u, v) \\
        & = (n - \deg(u)) + (\deg^+(u) - \deg^+(v)) + \deg(u, v) \\
        & \leq - \tau - r + \deg(u, v). \qedhere
    \end{align*}
\end{proof}

In this language it is easy to give a proof of Graver's result that $\tau = 1$ suffices to guarantee the presence of a triangle. Let $u$ be a vertex of $G$ with $\deg^+(u)$ minimal. Since $\delta(G) \geq n + 1$, $u$ has at least one forward neighbour $v$. Then, by minimality of $\deg^+(u)$, the edge $uv$ is a booster and so, by \cref{lem:booster}, $\deg(u, v) \geq \tau = 1$ which gives a triangle in $G$.

The following is a useful structural lemma that we shall frequently apply. For the reader it may be helpful to think of $H$ as being $G[A, B]$ -- there is only one instance where we need to apply the lemma to a strict subgraph of $G[A, B]$.

\begin{lem}\label{lem:structure}
    For every positive integer $t$ and constant $\lambda > 0$ there is a constant $K$ such that the following holds. Let $G$ be an $n \times n \times n$ tripartite graph with parts $V_1, V_2, V_3$ and $\delta(G) \geq n + \tau$. Let $A \subset V_1$, $B \subset V_2$, $C \subset V_3$, and $H \subset G[A, B]$ be such that the density of $H[A, B]$ is at least $\lambda$ and $\lambda \abs{A} \geq t$. Then either $G[A, B, C]$ contains $K_{t, t, t}$ or there exist $B^\ast \subset B$ and $C^\ast \subset C$ such that
    \begin{itemize}[noitemsep]
        \item $\abs{C^\ast} \leq t \cdot \max_{a \in A} \deg^+(a)$,
        \item $\abs{B^\ast} \geq \frac{1}{2}(\lambda/e)^t \abs{B} - \abs{C}^{1/t}$,
        \item for all $b \in B^\ast$\textup{:} $e(b, C^\ast) \geq e(b, C) - K \cdot \abs{C}^{1 - 1/t}$,
        \item the density of $H[A, B^\ast]$ is at least $\frac{1}{2} \lambda$.
    \end{itemize}
\end{lem}

\begin{proof}
    Applying \cref{lem:coneigh_density} to $H[A, B]$ gives distinct vertices $a_1, \dotsc, a_t \in A$ and $B_1 \subset B \cap N_H(a_1, \dotsc, a_t)$ such that $\abs{B_1} \geq \frac{1}{2}(\lambda/e)^t \abs{B}$ and $e_H(b, A) \geq \frac{1}{2} \lambda \abs{A}$ for every $b \in B_1$.

    Let $C^\ast = C \setminus N(a_1, \dotsc, a_t)$. Note that $C^\ast \subset \bigcup_i (V_3 \setminus N(a_i))$ and so
    \begin{equation*}
        \abs{C^\ast} \leq \sum_i \abs{V_3 \setminus N(a_i)} = \sum_i (n - \deg^-(a_i)) \leq \sum_i \deg^+(a_i) \leq t \cdot \max_{a \in A} \deg^+(a).
    \end{equation*}
    Let $B_2$ be those vertices in $B_1$ with at least $K \cdot \abs{C}^{1 - 1/t}$ neighbours in $C \setminus C^\ast$. If $\abs{B_2} \geq \abs{C}^{1/t}$, then
    \begin{align*}
        e(B_2, C \setminus C^\ast) \geq K \abs{B_2} \cdot \abs{C}^{1 - 1/t} \geq K/2 \cdot (\abs{B_2} \cdot \abs{C}^{1 - 1/t} + \abs{C}).
    \end{align*}
    Provided $K$ is sufficiently large, \cref{thm:KST} gives a $K_{t, t}$ between $B_2 \subset N_H(a_1, \dotsc, a_t) \subset N(a_1, \dotsc, a_t)$ and $C \setminus C^\ast = C \cap N(a_1, \dotsc, a_t)$. Together with $a_1, \dotsc, a_t$ this gives a $K_{t, t, t}$ in $G[A, B, C]$. Thus $\abs{B_2} < \abs{C}^{1/t}$. Let $B^\ast = B_1 \setminus B_2$. Then $\abs{B^\ast} \geq \abs{B_1} - \abs{C}^{1/t} \geq \frac{1}{2}(\lambda/e)^t \abs{B} - \abs{C}^{1/t}$ and every vertex of $B^\ast$ has fewer than $K \cdot \abs{C}^{1 - 1/t}$ neighbours in $C \setminus C^\ast$. Finally, every $b \in B^\ast \subset B_1$ satisfies $e_H(b, A) \geq \frac{1}{2} \lambda \abs{A}$ and so the density of $H[A, B^\ast]$ is at least $\frac{1}{2} \lambda$.
\end{proof}

\begin{rmk}\label{rmk:structure}
    In most cases we will use \cref{lem:structure} in the form stated above. However, there is one instance where we will want to additionally recall that there were $t$ vertices $a_1, \dotsc, a_t \in A$ such that $B^\ast \subset N_H(a_1, \dotsc, a_t)$, $C^\ast = C \setminus N(a_1, \dotsc, a_t)$, and every vertex in $B^\ast$ has at least $\frac{1}{2} \lambda \abs{A}$ neighbours in $A$.
\end{rmk}

\section{Proof of \texorpdfstring{\cref{thm:main}}{Theorem 1.3}}\label{sec:main_thm}

Each booster $uv$ gives $\binom{\tau}{t}$ copies of $K_{1, 1, t}$ (where the two singleton vertices are $u$ and $v$). Thus, if there are sufficiently many booster edges, then \cref{lem:K11t} would give a copy of $K_{t, t, t}$. In general there are not enough booster edges to make this simple counting argument work and we will instead need to pin down the structure of the booster edges within $G$.

We will now outline the proof. We say an edge $uv$ is \defn{$r$-heavy} if $\deg(u, v) \geq r$. Note that booster edges are $\tau$-heavy and $r$-boosters are $(r + \tau)$-heavy by \cref{lem:booster}.
We start with an $n \times n \times n$ tripartite graph $G$ with minimum degree $n + \tau$.
We first show (\cref{lem:initialconfig}) that there is some $k = \Omega(\tau)$ such that one of the following two structures must be present in $G$:
\begin{enumerate}[label = (\roman{*})]
    \item \textbf{Dense subgraph of boosters:} a bipartite graph of booster edges with parts of size $\Omega(k)$ which has edge density $1/200$
    ;
    \item \textbf{Many heavy edges:} $\Omega(k)$ vertices each having forward degree $\cO(k)$ and each being incident to $\Omega(k)$ forward edges that are $\Omega(k)$-heavy.
\end{enumerate}

We then address cases (i) and (ii) in \cref{sec:densebooster,sec:heavy}, respectively.
We first show (\cref{lem:denseboosterwin}) that if case (i) occurs, then there is a copy of $K_{t, t, t}$ in $G$. We next show (\cref{lem:lotsheavywin}) that if case (ii) occurs, then there is a copy of $K_{t, t, t}$ in $G$. Finally, combining \cref{lem:initialconfig,lem:denseboosterwin,lem:lotsheavywin} yields \cref{thm:main}. The proofs of \cref{lem:denseboosterwin,lem:lotsheavywin} both use density increment, albeit in different ways.

\Cref{fig:initial-config} shows the two possible configurations arising from \cref{lem:initialconfig}. For (ii), since the vertices of $S$ have forward degree $\cO(k)$, they must have backward degree $n - \cO(k)$, as is displayed in the figure.

\begin{figure}[H]
    \centering
    \begin{subfigure}{0.45\textwidth}
        \centering
        \includegraphics[width = 6cm]{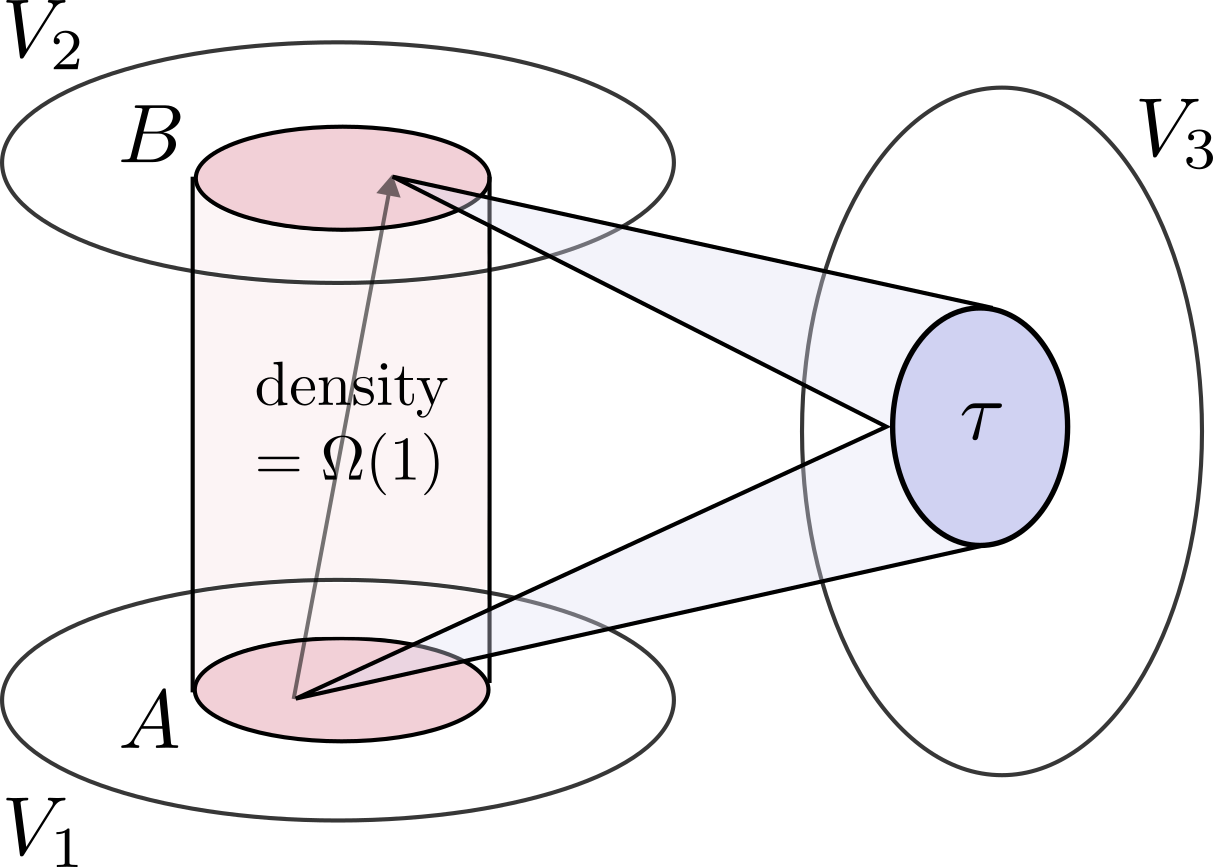}
        \subcaption{Dense subgraph of boosters}
    \end{subfigure}
    \begin{subfigure}{0.45\textwidth}
        \centering
        \includegraphics[width = 6cm]{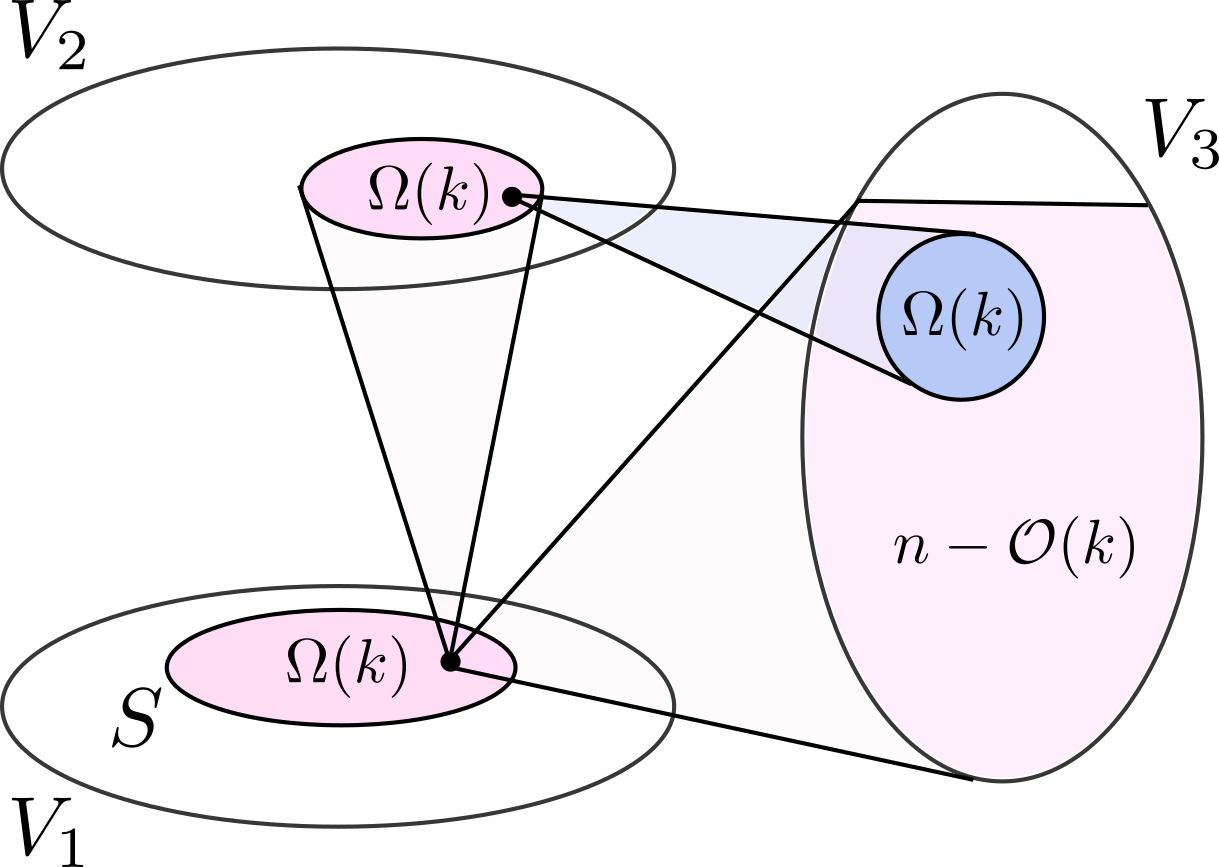}
        \subcaption{Many heavy edges}
    \end{subfigure}
    \caption{The configurations arising from \cref{lem:initialconfig}.}\label{fig:initial-config}
\end{figure}

\subsection{Finding the initial configuration}\label{sec:initialconfig}

Here we show that $G$ must contain one of two configurations shown in \cref{fig:initial-config}: either a dense subgraph of booster edges or many $\Omega(k)$-heavy edges. 
The constants appearing in \cref{lem:initialconfig} are not particularly important and we have not optimised them.
What matters is that the density of the bipartite graph of booster edges in (i) is $\Omega(1)$, the forward degrees in (ii) are $\cO(k)$, and all other quantities are $\Omega(k)$.
Throughout the proof we use the fact that if $uv$ is a forward edge and $\deg^+(v) \geq \deg^+(u) + r$, then $uv$ is an $r$-booster (and so is $r$-heavy).

\begin{lem}\label{lem:initialconfig}
    Let $G$ be an $n \times n \times n$ tripartite graph with $\delta(G) \geq n + \tau$. There is an integer $k \geq \tau/100$ and $i \in \set{1, 2, 3}$ such that $G$ contains either
    \begin{enumerate}[label=\textup{(}\roman{*}\textup{)}]
        \item a subgraph of booster edges \textup{(}in $G[V_i, V_{i + 1}]$\textup{)} with parts of size at least $k/200$ and density at least $1/200$, or
        \item $k/12$ vertices in $V_i$ each having forward degree at most $100k$ and each being incident to at least $k/200$ forward edges that are $k/100$-heavy.
    \end{enumerate}
\end{lem}

\begin{proof}
    Order the vertices of $G$ as $v_1, \dotsc, v_{3n}$ where $\deg^+(v_1) \leq \dotsb \leq \deg^+(v_{3n})$.
    Let $k \coloneqq \min\set{i \in \bZ^+ \colon \deg^+(v_i) \leq 98i}$ (the set is non-empty since $\deg^+(v_{3n}) \leq n \leq 294 n$). 
    Since $n + \tau \leq \deg^-(v_k) + \deg^+(v_k) \leq \deg^+(v_k) + n$, we have $\tau \leq \deg^+(v_k) \leq 98k$ and so $k \geq \tau/100$, as promised. Finally, by minimality of $k$, $\deg^+(v_k) \geq \deg^+(v_{k - 1}) > 98(k - 1) \geq 95k$.
        
    Without loss of generality, at least one third of $v_{k/2}, \dotsc, v_k$ lie in $V_1$. Call this set $A$ so $\abs{A} \geq k/6$. By minimality of $k$, $\deg^+(v_{k/2}) \geq 98 \cdot k/2 = 49k$ and so every vertex of $A$ has forward degree between $49k$ and $98k$.

    Define the vertex-sets $B$ and $C$ as follows:
    \begin{align*}
        B & \coloneqq \set{v \in V_2 \colon \deg^+(v_k) \leq \deg^+(v) \leq \deg^+(v_k) + k/100}, \\
        C & \coloneqq \set{v \in V_3 \colon \deg^+(v_k) \leq \deg^+(v) \leq \deg^+(v_k) + k/50}.
    \end{align*}
    Note that every vertex of $A \cup B \cup C$ has forward degree at most $\deg^+(v_k) + k/50 < 100k$. This will ensure the forward degree condition of (ii) always holds. Also note, by definition, that every edge between $A$ and $B$ is a booster. Finally, every vertex of $A$ has at least $49k$ neighbours in $V_2$ and at most $k$ of these neighbours have forward degree at most $\deg^+(v_k)$ (by the vertex-ordering). Hence, every vertex of $A$ sends at least $48k$ booster edges to $V_2$.

    Suppose there is a set $X$ of $k/12$ vertices in $A$ such that every vertex of $X$ has at least $10k$ neighbours in $V_2 \setminus B$. Now, at most $k$ vertices have forward degree less than $\deg^+(v_k)$ and so every vertex of $X$ has $9k$ neighbours in $V_2 \setminus B$ whose forward degree is at least $\deg^+(v_k)$ and so is greater than $\deg^+(v_k) + k/100$ by the definition of $B$. All the resulting edges are $k/100$-boosters and so $X$ satisfies (ii). Similarly, if there is a set $Y$ of $k/12$ vertices in $B$ such that every vertex of $Y$ has at least $10k$ neighbours in $V_3 \setminus C$, then $Y$ satisfies (ii). Thus we may and will assume that there is a set $A' \subset A$ of at least $\abs{A} - k/12 \geq k/12$ vertices all of which have fewer than $10k$ neighbours in $V_2 \setminus B$ and a set of vertices $B' \subset B$ of size at least $\abs{B} - k/12$ all of which have fewer than $10k$ neighbours in $V_3 \setminus C$. Now, since every vertex of $A$ has at least $49k$ neighbours in $V_2$ and $\abs{B \setminus B'} \leq k/12$, every vertex of $A'$ has at least $49k - 10k - k/12 > 30k$ neighbours in $B'$. Also every vertex of $B'$ has at least $\deg^+(v_k) > 95k$ neighbours in $V_3$ and so at least $80k$ neighbours in $C$. In summary,
    
    \begin{itemize}
        \item $\abs{A'} \geq k/12$ and for all $a \in A'$: $e(a, B') \geq 30k$ and $49k \leq \deg^+(a) \leq 98k$,
        \item $\abs{B'} \geq 30k$ and for all $b \in B'$: $e(b, C) \geq 80k$ and $\deg^+(v_k) \leq \deg^+(b) \leq \deg^+(v_k) + k/100$,
        \item $\abs{C} \geq 80k$ and for all $c \in C$: $\deg^+(v_k) \leq \deg^+(c) \leq \deg^+(v_k) + k/50$,
        \item all edges between $A'$ and $B'$ are boosters,
        \item all vertices of $A' \cup B' \cup C$ have forward degree at most $100k$,
        \item every vertex of $A'$ sends at least $48k$ booster edges to $V_2$.
    \end{itemize}
    
    Suppose there are $40k$ vertices in $C$ each of which has at least $k/100$ neighbours in $A'$. 
    If each of the $40k$ vertices is incident to at least $k/200$ $k$-heavy edges to $A'$, then they satisfy (ii). Otherwise, one of the $40k$ vertices, $c$, say, is incident to at least $k/200$ non-$k$-heavy edges to $A'$. Then there is a set $A'' \subset A'$ of size $k/200$ such that $\deg(a, c) \leq k$ for all $a \in A''$. But then every vertex of $A''$ sends at least $47k$ booster edges to $V_2 \setminus N(c)$. On the other hand, since $c \in C$, $\abs{V_2 \setminus N(c)} = n - \deg^-(c) \leq \deg^+(c) \leq 100k$. Then $A''$ and $V_2 \setminus N(c)$ give (i) with edge density greater than $47/100$.
    
    Thus we may and will assume that there is a set $C' \subset C$ of at least $\abs{C} - 40k \geq 40k$ vertices each of which has at most $k/100$ neighbours in $A'$. Since $\abs{C \setminus C'} \leq 40k$, every vertex of $B'$ still has at least $80k - 40k = 40k$ neighbours in $C'$. In summary,
    
    \begin{itemize}
        \item $\abs{A'} \geq k/12$ and for all $a \in A'$: $e(a, C') \geq 30k$ and $49k \leq \deg^+(a) \leq 98k$,
        \item $\abs{B'} \geq 30k$ and for all $b \in B'$: $e(b, C') \geq 40k$ and $\deg^+(v_k) \leq \deg^+(c) \leq \deg^+(v_k) + k/100$,
        \item $\abs{C'} \geq 40k$ and for all $c \in C'$: $e(c, A') \leq k/100$ and $\deg^+(v_k) \leq \deg^+(c) \leq \deg^+(v_k) + k/50$,
        \item all edges between $A'$ and $B'$ are boosters,
        \item all vertices of $A' \cup B' \cup C'$ have forward degree at most $100k$.
    \end{itemize}
    
    If the density between $A'$ and $B'$ is at least $1/200$, then they give (i). Thus at least half the vertices of $B'$ have fewer than $\abs{A'}/100$ neighbours in $A'$. Call this set $B''$: $\abs{B''} \geq 15k$ and for all $b \in B''$, $e(b, A') \leq \abs{A'}/100 \leq k/100$.
    
    Consider any $b \in B''$, $c \in C'$. We have
    \begin{align*}
        e(b, V_1 \setminus A') & = \deg^-(b) - e(b, A') \geq n - \deg^+(b) - k/100 \geq n - \deg^+(v_k) - k/50, \\
        e(c, V_1 \setminus A') & = \deg^+(c) - e(c, A') \geq \deg^+(c) - k/100 \geq \deg^+(v_k) - k/100.
    \end{align*}
    On the other hand, $\abs{V_1 \setminus A'} = n - \abs{A'} \leq n - k/12$ and so, 
    \begin{align*}
        \deg(b, c) & \geq e(b, V_1 \setminus A') + e(c, V_1 \setminus A') - \abs{V_1 \setminus A'} \\
        & \geq (n - \deg^+(v_k) - k/50) + (\deg^+(v_k) - k/100) - (n - k/12) \\
        & = k/12 - 3k/100 \geq k/20,
    \end{align*}
    and so every edge between $B''$ and $C'$ is $k/20$-heavy. But $\abs{B''} \geq 15k$ and every vertex of $B'' \subset B'$ has at least $40k$ neighbours in $C'$ so $B''$ satisfies (ii).
\end{proof}

\subsection{Dense regime}\label{sec:densebooster}

In this section we show that if case (i) of \cref{lem:initialconfig} occurs, that is, $G$ contains a constant density subgraph $H = (A, B; E)$ of booster edges, then $G$ must contain $K_{t, t, t}$. Indeed, note that, if $\delta(G) \geq n + \tau$ where $\tau \geq Kn^{1-1/t}$ and $K/20000$ is larger than the constant $K_{\ref{lem:denseboosterwin}}$ in the lemma below, then case (i) of \cref{lem:initialconfig} satisfies its assumptions with $1/200$ playing the role of $\lambda$ (after relabelling the parts if necessary).
\begin{lem}\label{lem:denseboosterwin}
    For every $\lambda \in (0, 1]$ and integer $t \geq 2$ there is a constant $K$ such that the following holds. Suppose that $G$ is an $n \times n \times n$ tripartite graph with $\delta(G) = n + \tau$, where $\tau \geq K n^{1 - 1/t}$, and there are vertex-sets $A \subset V_1$, $B \subset V_2$, and a subgraph $H \subset G[A, B]$ such that
    \begin{itemize}
        \item $\abs{A}, \abs{B} \geq K n^{1 - 1/t}$,
        \item every edge of $H$ is a booster,
        \item $e(H) \geq \lambda \abs{A} \abs{B}$.
    \end{itemize}
    Then $G$ contains $K_{t, t, t}$.
\end{lem}

We will prove this using the following ``density increment'' result. It shows that if $G$ is $K_{t, t, t}$-free and every edge $ab$ of $H$ has codegree at most some bound $d$, then $H$ has a subgraph $H'$, still of constant density, in which all edges have codegree $\cO(d \cdot \abs{A}^{-1/t^2} + n^{1 - 1/t})$. This will be iterated until the codegree of all the edges in the subgraph is $\cO(n^{1 - 1/t})$. However, every edge of $H$ is a booster and so has codegree $\Omega(n^{1 - 1/t})$ by \cref{lem:booster}. Provided the constants are chosen correctly this gives a contradiction and so $G$ must contain $K_{t, t, t}$, which proves \cref{lem:denseboosterwin}.

\begin{prop}\label{prop:denseboosterincrement}
    For every $\lambda \in (0, 1]$ and integer $t \geq 2$ there is a constant $K$ such that the following holds. Let $G$ be an $n \times n \times n$ tripartite graph with $\delta(G) \geq n + \tau$ and suppose there is $d > 0$, vertex-sets $A \subset V_1$, $B \subset V_2$, and a subgraph $H \subset G[A, B]$ such that
    \begin{itemize}[noitemsep]
        \item $\abs{A}, \abs{B} \geq K n^{1 - 1/t}$,
        \item every edge of $H$ is a booster,
        \item the density of $H[A, B]$ is at least $\lambda$,
        \item $\deg(a, b) \leq d$ for every edge $ab$ of $H$.
    \end{itemize}
    Then either $G$ contains $K_{t, t, t}$ or there is a subgraph $H'$ of $H$ such that
    \begin{itemize}[noitemsep]
        \item the density of $H'[A, B]$ is at least $\lambda^{t + 3}/(1024 e^t)$,
        \item $\deg(a, b) \leq K(d \cdot \abs{A}^{-1/t^2} + n^{1 - 1/t})$ for every edge $ab$ of $H'$.
    \end{itemize}
\end{prop}

Let us briefly sketch the proof of \cref{prop:denseboosterincrement}. 
The key observation is that for any booster $ab \in H[A,B]$ with 
$\deg(a,b) \le d$, all but at most $2d$ vertices of $V_3$ are adjacent to
exactly one of $a$ and $b$ (by \cref{lem:booster}). Iterating this
observation shows that pairs of vertices connected by a short path in
$H$ have closely related neighbourhoods in $V_3$: vertices lying in
different parts of $H[A, B]$ have approximately complementary
neighbourhoods in $V_3$, whereas vertices in the same part have
neighbourhoods that are nearly identical. 

We make use of this observation in the following way. First, we select subsets $A^\ast \subseteq A$ and
$B^\ast \subseteq B$ such that the induced graph $H[A^\ast, B^\ast]$ has small
diameter while retaining positive edge density. Applying the observation to $H[A^\ast, B^\ast]$, we then find disjoint sets
$U_1, U_2 \subseteq V_3$ with the following properties: the bipartite
graphs $G[A^\ast, V_3 \setminus (U_1 \cup U_2)]$ and $G[B^\ast, U_1]$ are
both almost complete, and $\abs{U_2} = \cO(d)$.\footnote{For technical reasons, the proof further requires that $B^\ast \subseteq N(a_1, \dots, a_t)$ for some carefully selected $a_1, \dots,a_t \in A$, and that $V_3 \setminus N(a_1, \dots, a_t) = U_1 \cup U_2$, which translates to a negligible loss in density.}

Under these circumstances, if a large proportion of the edges in
$H[A^\ast, B^\ast]$ have most of their common neighbourhood contained in either
$U_1$ or $V_3 \setminus (U_1 \cup U_2)$, then the density of
$H[A^\ast, B^\ast]$ readily yields a copy of $K_{t, t, t}$. Consequently, a
substantial fraction of the edges of $H[A^\ast, B^\ast]$ must instead have
most of their common neighbourhood contained in $U_2$. If the common neighbourhood in $U_2$ of a substantial fraction of such edges has large density in $U_2$, then a careful counting argument gives a $K_{t, t, t}$. Otherwise, there is a reasonable fraction of edges of $H[A^\ast, B^\ast]$ whose common neighbourhood is mostly in $U_2$ and further is small compared to $U_2$, which gives the second outcome.

\begin{proof}[Proof of \cref{prop:denseboosterincrement}]
    Since every edge of $H$ is a booster, \cref{lem:booster} implies that every edge $ab$ of $H$ satisfies
    \begin{equation}\label{eq:codegd}
        \abs{V_3 \setminus (N(a) \cup N(b))} \leq \deg(a, b) \leq d.
    \end{equation}
    Applying \cref{rmk:structure} to $H$ with $C = V_3$ gives $t$ distinct vertices $a_1, \dotsc, a_t \in A$ and a set $B' \subset N_H(a_1, \dotsc, a_t)$ such that
    \begin{enumerate}[noitemsep,label=(\alph{*})]
        \item $\abs{B'} \geq \frac{1}{2}(\lambda/e)^t \cdot \abs{B} - n^{1/t}$,
        \item every $b \in B'$ has $e_H(b, A) \geq \frac{1}{2} \lambda \abs{A}$,
        \item every $b \in B'$ has at most $K_{\ref{lem:structure}} n^{1 - 1/t}$ neighbours in $N(a_1, \dotsc, a_t) \cap V_3$,
    \end{enumerate}
    where the constant $K_{\ref{lem:structure}}$ given by the lemma depends only on $t$ and $\lambda$. Since $n^{1 - 1/t} \geq n^{1/t}$, (a) implies $\abs{B'} \geq \frac{1}{4}(\lambda/e)^t \cdot \abs{B}$, provided $K$ is large enough. Also, by (b), $e(H[A, B']) \allowbreak \geq \lambda/2 \cdot \abs{A} \abs{B'}$.

    We now do the following sparsification procedure on $H[A, B']$. Repeatedly delete vertices of $B'$ with degree less than $\lambda/4 \cdot \abs{A}$ and vertices of $A$ with degree less than $\lambda/4 \cdot \abs{B'}$. In total this removes fewer than $\lambda/2 \cdot \abs{A} \abs{B'}$ edges from $H[A, B']$ and so some edges remain. Let the resulting vertex-sets be $A^\ast \subset A$ and $B^\ast \subset B'$. Note that every vertex of $A^\ast$ has at least $\lambda/4 \cdot \abs{B'}$ neighbours in $B^\ast$ and every vertex of $B^\ast$ has at least $\lambda/4 \cdot \abs{A}$ neighbours in $A^\ast$. In particular, $\abs{A^\ast} \geq \lambda/4 \cdot \abs{A}$ and $\abs{B^\ast} \geq \lambda/4 \cdot \abs{B'}$.

    We now partition $V_3$ into $N(a_1, \dotsc, a_t) \cap V_3$, $U_1$, and $U_2$ where
    \begin{align*}
        U_1 & \coloneqq \set{v \in V_3 \setminus N(a_1, \dotsc, a_t) \colon \abs{N(v) \cap B^\ast} \geq (1 - \lambda/8) \abs{B^\ast}}, \\
        U_2 & \coloneqq \set{v \in V_3 \setminus N(a_1, \dotsc, a_t) \colon \abs{N(v) \cap B^\ast} < (1 - \lambda/8) \abs{B^\ast}}.
    \end{align*}
    Fix $b \in B^\ast$. Note that
    \begin{equation*}
        V_3 \setminus (N(a_1, \dotsc, a_t) \cup N(b)) = \bigcup_{i = 1}^t (V_3 \setminus (N(a_i) \cup N(b)))
    \end{equation*}
    and so applying inequality~\eqref{eq:codegd} to the edges $a_i b$ shows that $b$ has at most $td$ non-neighbours in $V_3 \setminus N(a_1, \dotsc, a_t)$. In particular, the number of non-edges between $B^\ast$ and $U_2$ is at most $td \cdot \abs{B^\ast}$. However, each vertex of $U_2$ contributes $\lambda/8 \cdot \abs{B^\ast}$ to this tally and so $\abs{U_2} \leq 8 t d \lambda^{-1}$.
    
    Consider a pair $a \in A^\ast$ and $v \in U_1$. By the degree conditions on $H[A^\ast, B^\ast]$, $a$ has at least $\lambda/4 \cdot \abs{B'} \geq \lambda/4 \cdot \abs{B^\ast}$ neighbours in $B^\ast$. On the other hand, $v \in U_1$ has at most $\lambda/8 \cdot \abs{B^\ast}$ non-neighbours in $B^\ast$ and so $a$ and $v$ have at least $\lambda/8 \cdot \abs{B^\ast}$ common neighbours in $B^\ast$.

    Let $K_{\ref{cor:K11tdense}}$ be the constant obtained when applying \cref{cor:K11tdense} to $t$ and $\lambda/8$.
    If half the vertices of $A^\ast$ have at least $2 K_{\ref{cor:K11tdense}} n^{1 - 1/t}$ neighbours in $U_1$, then
    \begin{equation*}
        e(G[A^\ast, U_1]) \geq K_{\ref{cor:K11tdense}} n^{1 - 1/t} \abs{A^\ast}.
    \end{equation*}
    Also, provided $K$ is sufficiently large, $\abs{A^\ast} \geq n^{1 - 1/t} \geq n^{1/t}$ and $\lambda/8 \cdot \abs{B^\ast} \geq t$ and so \cref{cor:K11tdense} implies that $G[A^\ast, B^\ast, U_1]$ contains $K_{t, t, t}$.
    Hence, we may and will assume that at least half the vertices of $A^\ast$ have fewer than $2K_{\ref{cor:K11tdense}} n^{1 - 1/t}$ neighbours in $U_1$. Let $A^{\ast\ast}$ be the set of such vertices. By the minimum degree conditions on $H[A^\ast, B^\ast]$, the number of edges in $H[A^{\ast\ast}, B^{\ast}]$ is at least $\abs{A^{\ast\ast}} \cdot \lambda/4 \cdot \abs{B^\ast}$.

    Let $K_{\ref{lem:K11t}}$ be the constant obtained when applying \cref{lem:K11t} to $t$.
    Call an edge $ab$ of $H[A^{\ast\ast}, B^{\ast}]$ \defn{sparse} if $\abs{N(a, b) \cap U_2} \leq S \coloneqq \max \set{64 K_{\ref{lem:K11t}} t^2 d \lambda^{-2} (\lambda/8 \cdot \abs{A})^{-1/t^2}, t}$. Suppose that at least half the edges of $H[A^{\ast\ast}, B^{\ast}]$ are not sparse. Then the number of $K_{1, 1, t}$s in $A^{\ast\ast} \times B^\ast \times U_2$ is at least
    \begin{align*}
        \lambda/8 \cdot \abs{A^{\ast\ast}} \cdot \abs{B^{\ast}} \cdot \tbinom{S}{t} & \geq \lambda/8 \cdot \abs{A^{\ast\ast}} \cdot \abs{B^{\ast}} \cdot (S/t)^t \\
        & \geq \lambda/8 \cdot \abs{A^{\ast\ast}} \cdot \abs{B^{\ast}} \cdot (64 K_{\ref{lem:K11t}} t d \lambda^{-2} (\lambda/8 \cdot \abs{A})^{-1/t^2})^t \\
        & \geq \lambda/8 \cdot \abs{A^{\ast\ast}} \cdot \abs{B^{\ast}} \cdot (64 K_{\ref{lem:K11t}} t d \lambda^{-2} \abs{A^{\ast\ast}}^{-1/t^2})^t \\
        & \geq (8td \lambda^{-1})^t \cdot K_{\ref{lem:K11t}} \cdot \abs{B^{\ast}} \cdot \abs{A^{\ast\ast}}^{1 - 1/t} \\
        & \geq K_{\ref{lem:K11t}} \cdot \abs{B^{\ast}} \cdot \abs{A^{\ast\ast}}^{1 - 1/t} \cdot  \abs{U_2}^t,
    \end{align*}
    where the first inequality follows from $\binom{x}{t} \geq (x/t)^t$ for all $x \geq t$, the third inequality follows from $\abs{A^{\ast\ast}} \geq 1/2 \cdot \abs{A^\ast} \geq \lambda/8 \cdot \abs{A}$, and the last inequality uses the bound $\abs{U_2} \leq 8 td \lambda^{-1}$ proved above.
    Also, provided $K$ is sufficiently large, $\abs{B^\ast} \geq n^{1 - 1/t} \geq n^{1/t}$ and hence \cref{lem:K11t} implies that there is a $K_{t, t, t}$ in $G[A^{\ast\ast}, B^{\ast}, U_2]$. 
    Thus we may and will assume that at least half the edges of $H[A^{\ast\ast}, B^{\ast}]$ are sparse. Let $H'$ be the subgraph consisting of sparse edges. For every edge $ab$ of $H'$: 
    \begin{itemize}
        \item $a \in A^{\ast\ast}$ has fewer than $2K_{\ref{cor:K11tdense}} n^{1 - 1/t}$ neighbours in $U_1$,
        \item $b \in B^{\ast} \subset B'$ has at most $K_{\ref{lem:structure}} n^{1 - 1/t}$ neighbours in $N(a_1, \dotsc, a_t) \cap V_3$,
        \item $a$ and $b$ have at most $S$ common neighbours in $U_2$.
    \end{itemize}
    Overall, $a$ and $b$ have at most $(K_{\ref{lem:structure}} + 2K_{\ref{cor:K11tdense}}) n^{1 - 1/t} + S$ common neighbours. This is at most $K(d \cdot \abs{A}^{-1/t^2} + n^{1 - 1/t})$ for sufficiently large $K$. Finally,
    \begin{align*}
        e(H') & \geq 1/2 \cdot e(H[A^{\ast\ast}, B^{\ast}]) \geq \lambda/8 \cdot \abs{A^{\ast\ast}} \abs{B^{\ast}} \geq \lambda^3/256 \cdot \abs{A} \abs{B'} \\
        & \geq \lambda^{t + 3}/(1024 e^t) \cdot \abs{A} \abs{B},
    \end{align*}
    as required.
\end{proof}

We are now ready to show that if there is a dense subgraph of booster edges in $G$ (i.e.\ if case (i) of \cref{lem:initialconfig} occurs), then $G$ contains $K_{t, t, t}$.

\begin{proof}[Proof of \cref{lem:denseboosterwin}]
    Suppose for a contradiction that $G$ does not contain $K_{t, t, t}$. We will repeatedly apply \cref{prop:denseboosterincrement}: note that $\abs{B} \geq n^{1 - 1/t} \geq n^{1/2}$ and so $d \cdot \abs{B}^{-1/t^2} \leq d n^{-1/2t^2}$. Let $K_{\ref{prop:denseboosterincrement}}$ be the constant obtained in \cref{prop:denseboosterincrement} when it is applied to $\lambda$ and $t$.
    Provided $K \geq K_{\ref{prop:denseboosterincrement}}$, the graph $H$ and vertex-sets $A, B$ satisfy the hypotheses of \cref{prop:denseboosterincrement} with $d = n$. This gives a subgraph $H'$ of $H$ with edge density $e_\lambda = \Omega(\lambda^{t + 3})$ and such that, for every edge $ab$ of $H'$, $\deg(a, b) \leq K_{\ref{prop:denseboosterincrement}}(d n^{-1/2t^2} + n^{1 - 1/t}) \leq c_\lambda n^{1 - 1/2t^2}$ where $c_\lambda$ and $e_\lambda$ are positive constants depending on $\lambda$ and $t$. 

    Now let $K_{\ref{prop:denseboosterincrement}}$ be the constant obtained when \cref{prop:denseboosterincrement} is applied to $e_{\lambda}$ and $t$. Provided $K \geq K_{\ref{prop:denseboosterincrement}}$, the graph $H'$ and vertex-sets $A$, $B$ again satisfy the hypotheses of \cref{prop:denseboosterincrement} now with $d = c_\lambda n^{1 - 1/2t^2}$. This gives a subgraph $H''$ of $H$ with edge density $e'_{\lambda} = \Omega(e_\lambda^{t + 3})$ and such that, for every edge $ab$ of $H''$, $\deg(a, b) \leq K_{\ref{prop:denseboosterincrement}}(dn^{-1/2t^2} + n^{1 - 1/t}) \leq c'_{\lambda} n^{1 - 2/2t^2}$ where $c'_{\lambda}$ and $e'_{\lambda}$ are positive constants depending on $\lambda$ and $t$.
    
    We may repeat this argument $2t - 2$ more times to obtain, provided $K$ is sufficiently large in terms of $\lambda$ and $t$, a subgraph $H^\ast$ of $H$ with edge density at least $e^\ast_\lambda$ and where every edge $ab$ of $H^\ast$ satisfies $\deg(a, b) \leq c^\ast_\lambda n^{1 - 1/t}$.

    Now, every edge of $H^\ast$ is a booster (since every edge of $H$ is). This means that every edge $ab$ of $H^\ast$ satisfies $\deg(a, b) \geq \tau \geq Kn^{1 - 1/t}$. Therefore, if $K \geq c^{\ast}_{\lambda}$, then we obtain a contradiction. Thus, if $K$ is sufficiently large in terms of $\lambda$ and $t$, then $G$ must contain $K_{t, t, t}$.
\end{proof}

This concludes the argument when case (i) from \cref{lem:initialconfig} arises. While we are discussing the dense regime we will note one further dense configuration that guarantees a $K_{t, t, t}$. This will be useful for case (ii). Note that this lemma does not follow immediately from \cref{lem:denseboosterwin} as the edges are heavy rather than boosters.

\begin{lem}\label{lem:denseheavywin}
    For every $\lambda \in (0, 1]$ and integer $t \geq 2$ there is a constant $K$ such that the following holds. Suppose that $G$ is an $n \times n \times n$ tripartite graph with $\delta(G) \geq n + \tau$, there is an integer $r \geq Kn^{1 - 1/t}$, vertex-sets $A \subset V_1$, $B \subset V_2$, and a subgraph $H \subseteq G[A, B]$ such that
    \begin{itemize}[noitemsep]
        \item $\abs{A} \geq K$, $\abs{B} \geq Kn^{1 - 1/t}$,
        \item every edge of $H$ is $r$-heavy,
        \item the density of $H[A, B]$ is at least $\lambda$,
        \item Every vertex of $A$ has forward degree at most $10^4 r$.
    \end{itemize}
    Then $G$ contains $K_{t, t, t}$.
\end{lem}

\begin{proof}
    We first apply our structural lemma, \cref{lem:structure}, to $A$, $B$, and $C = V_3$ (note that $\lambda \abs{A} \geq t$ provided $K$ is large). Either $G$ contains $K_{t, t, t}$ or there exists $B^\ast \subset B$ and $C^\ast \subset V_3$ such that
    \begin{itemize}[noitemsep]
        \item $\abs{C^\ast} \leq 10^4 tr$,
        \item $\abs{B^\ast} \geq \frac{1}{2} (\lambda/e)^t \abs{B} - n^{1/t}$,
        \item for all $b \in B^\ast$\textup{:} $e(b, C^\ast) \geq e(b, V_3) - K_{\ref{lem:structure}} \cdot n^{1 - 1/t}$,
        \item the density of $H[A, B^\ast]$ is at least $\frac{1}{2} \lambda$,
    \end{itemize}
    where the constant $K_{\ref{lem:structure}}$ is given by the lemma and depends only on $\lambda$ and $t$. Provided $K$ is large enough, the second bullet point implies that $\abs{B^\ast} \geq n^{1/t} \geq \abs{A}^{1/t}$.
    
    Consider an edge $ab \in E(H[A, B^\ast])$. It is $r$-heavy and so the third bullet point implies that $a$ and $b$ have at least $r - K_{\ref{lem:structure}} \cdot n^{1 - 1/t} \ge r/2$ common neighbours in $C^\ast$, provided $K \geq 2K_{\ref{lem:structure}}$. Hence every edge of $H[A, B^\ast]$ has at least $r/2 \geq \max\set{\abs{C^\ast}/(2t \cdot 10^4), t}$ common neighbours in $C^\ast$. Note that
    \begin{align*}
        e(G[A, B^\ast]) \geq e(H[A, B^\ast]) & \geq \tfrac{1}{2} \lambda \cdot \abs{A} \cdot \abs{B^\ast} \\
        & \geq \tfrac{1}{2} \lambda \cdot K^{1/t} \cdot \abs{A}^{1 - 1/t} \cdot \abs{B^\ast}.
    \end{align*}
    Thus, provided $K$ is sufficiently large, \cref{cor:K11tdense} implies that $G$ contains $K_{t, t, t}$.
\end{proof}

\subsection{Lots of heavy edges}\label{sec:heavy}

In this section we show that if case (ii) of \cref{lem:initialconfig} occurs, that is, $G$ contains $\Omega(k)$ vertices each having forward degree $\cO(k)$ and each being incident to $\Omega(k)$ forward edges that are $\Omega(k)$-heavy, then $G$ must contain $K_{t, t, t}$. 

\begin{lem}\label{lem:lotsheavywin}
    For every integer $t \geq 2$ there is a constant $K$ such that the following holds. Let $G$ be an $n \times n \times n$ tripartite graph with $\delta(G) \geq n + \tau$ where $\tau \geq Kn^{1 - 1/t}$. If case \textup{(}ii\textup{)} of \cref{lem:initialconfig} occurs for some $k \geq \tau/100$, then $G$ contains a $K_{t, t, t}$.
\end{lem}

We will first prove a result that converts the $k/100$-heavy edges of case (ii) into `big' boosters.

\begin{lem}\label{lem:squad}
    For every integer $t \geq 2$ there is a constant $K$ such that the following holds. Let $G$ be an $n \times n \times n$ tripartite graph with $\delta(G) \geq n + \tau$ where $\tau \geq Kn^{1 - 1/t}$. If case \textup{(}ii\textup{)} of \cref{lem:initialconfig} occurs for some $k \geq \tau/100$, then, either $G$ contains $K_{t, t, t}$ or there is some $r \geq 100k$ and $k/800$ vertices in a single part each having forward degree at most $r$ and each being incident to at least $k/800$ forward edges that are $2tr$-boosters.
\end{lem}

\begin{proof}
    Since case (ii) of \cref{lem:initialconfig} occurs we may assume, without loss of generality, that there is a set of $k/12$ vertices $A \subset V_1$ each having forward degree at most $100k$ and each being incident to at least $k/200$ forward edges that are $k/100$-heavy. Let $H$ be the subgraph of these $k/100$-heavy edges and let $B \subset V_2$ be the end-points of them in $V_2$. Partition $B$ into $B_1$ and $B_2$ where
    \begin{align*}
        B_1 \coloneqq \set{v \in B \colon \deg^+(v) \leq 300tk}, \\
        B_2 \coloneqq \set{v \in B \colon \deg^+(v) > 300tk}.
    \end{align*}
    Note that each edge between $A$ and $B_2$ is a $200tk$-booster. In particular, if $k/800$ of the vertices of $A$ have $k/800$ neighbours in $B_2$, then these vertices satisfy the conclusion of the lemma with $r = 100k$. Otherwise, there is a set $A' \subset A$ of at least $k/13$ vertices such that $e_H(a, B_1) \geq k/400$ for all $a \in A'$.

    Let
    \begin{equation*}
        C \coloneqq \set{c \in V_3 \colon \deg^+(c) \leq 900 t^2 k}.
    \end{equation*}
    Suppose that $\abs{C} \geq 200k$. Every vertex $a \in A' \subset A$ satisfies $\deg^-(a) = \deg(a) - \deg^+(a) \geq n - 100k$. Thus, every vertex of $A'$ has at most $100k$ non-neighbours in $V_3$ and so is adjacent to at least half of $C$. Thus $G[C, A']$ has density at least $1/2$. Applying the structural lemma, \cref{lem:structure}, to $G[C, A']$ with third part $B_1$ gives vertex-sets $A^\ast \subset A'$ and $B^\ast \subset B_1$ such that
    \begin{itemize}[noitemsep]
        \item $\abs{B^\ast} \leq t \cdot 900 t^2 k = 900 t^3k$,
        \item $\abs{A^\ast} \geq \frac{1}{2}(1/2e)^t \abs{A'} - n^{1/t}$,
        \item for all $a \in A^\ast$: $e(a, B^\ast) \geq e(a, B_1) - K_{\ref{lem:structure}} \cdot n^{1 - 1/t}$,
    \end{itemize}
    where the constant $K_{\ref{lem:structure}}$ depends only on $t$. We now consider $H[A^\ast, B^\ast]$ with the aim of applying \cref{lem:denseheavywin} with $r \coloneqq k/100$. 
    For every $a \in A^\ast$, $e_H(a, B^\ast) \geq e_H(a, B_1) - K_{\ref{lem:structure}} \cdot n^{1 - 1/t} \geq k/800$, provided $K$ is sufficiently large compared to $K_{\ref{lem:structure}}$ (note this also implies that $\abs{B^\ast} \geq k/800$). But then the first bullet point implies that the density $\lambda$ of $H[A^\ast, B^\ast]$ is at least $1/(720000t^3)$.
    Recall that $\abs{A'} \geq k/13$ and so the second bullet point implies that $\abs{A^\ast} \geq k/13^{t+1} \geq (K/13^{t+3}) \cdot n^{1 - 1/t} - n^{1/t} \geq (K/13^{t+4}) \cdot n^{1 - 1/t}$, provided $K$ is large enough. Also $\abs{B^\ast} \geq k/800$ and $r = k/100$ are both at least $K n^{1 - 1/t}/80000$.
    Finally, every vertex of $A^\ast \subseteq A$ has forward degree at most $100k = 10^4 r$ and every edge of $H[A^\ast, B^\ast]$ is $k/100$-heavy (since every edge of $H$ is). Thus, provided $K/13^{t+4} \geq K_{\ref{lem:denseheavywin}}$, applying \cref{lem:denseheavywin} with $1/(720000t^3)$ playing the role of $\lambda$ and $r$ playing its own role gives a $K_{t, t, t}$.

    Thus we may and will assume that $\abs{C} \leq 200k$. Call an edge of $H[A',B_1]$ \defn{concentrated} if it has at least $k/800$ common neighbours in $C$ and \defn{non-concentrated} otherwise: since every edge of $H$ is $k/100$-heavy, every non-concentrated edge has at least $k/800$ common neighbours in $V_3 \setminus C$. Suppose there is a vertex $a \in A'$ half of whose incident edges (in $H$) are not concentrated. This gives a set $B' \subset N_H(a) \cap B_1$ of size $k/800$ such that $a, b$ have at least $k/800$ common neighbours in $V_3 \setminus C$ for all $b \in B'$. Every vertex of $B'$ is in $B_1$ and so has forward degree at most $300tk$ while every vertex of $V_3 \setminus C$ has forward degree at least $900 t^2 k$ and so every vertex of $B'$ is incident to at least $k/800$ forward edges that are $600t^2k$-boosters. Hence $B'$ satisfies the conclusion of the lemma with $r = 300tk$. Otherwise at least half the edges of $H[A', B_1]$ are concentrated and so have $k/800 \geq \max\set{\abs{C}/160000, t}$ common neighbours in $C$. The number of such edges is at least
    \begin{equation*}
        \tfrac{1}{2} \cdot \abs{A'} \cdot k/400 \geq \frac{K}{80000} \cdot \abs{A'} \cdot n^{1 - 1/t}.
    \end{equation*}
    Also, provided $K$ is sufficiently large, $\abs{A'} \geq k/13 \geq n^{1/t}$ and hence \cref{cor:K11tdense} gives a $K_{t, t, t}$ in $G$.
\end{proof}

We are now ready to prove \cref{lem:lotsheavywin} and so finish the proof of \cref{thm:main}. Suppose that some $k \geq \tau/100$ as in \cref{lem:initialconfig} has been fixed. We define an \defn{$r$-squad} to be a set of $k/800$ vertices in a single part each having forward degree at most $r$ and each being incident to at least $k/800$ forward edges that are $2tr$-boosters. \Cref{lem:squad} guarantees that if case (ii) of \cref{lem:initialconfig} holds, then $G$ contains either $K_{t, t, t}$ or an $r$-squad for $r \geq 100k$. We take an $r$-squad with $r$ maximal and carry out a density increment argument relative to $r$ to show that in fact $G$ must contain $K_{t, t, t}$.

Since this is the most delicate argument in the paper, we now sketch its proof. Suppose that $G$ contains an $r$-squad $A \subseteq V_1$ with $r$ maximal. Let $H$ be the subgraph consisting of forward $2tr$-boosters incident with $A$, and set $B \coloneqq \bigcup_{v \in A} N_H(v) \subseteq V_2$. We begin by refining the pair $(A,B)$ via the following iterative procedure. As long as there exist a small subset $B' \subseteq B$ and a large subset $A' \subseteq A$ such that, for every $v \in A'$, the degree $e_H(v, B')$ is almost as large as $e_H(v, B)$, we replace $A$ by $A'$, $B$ by $B'$, and continue. When this procedure terminates, every vertex in $A$ still has degree $\Omega(k)$ into $B$ in $H$; additionally, the pair $(A,B)$ is now `nowhere dense', in the sense that no large subset of vertices in $A$ has its $H$-degrees concentrated into a small subset of $B$.

Next, let $C \subseteq V_3$ be the set of vertices $v$ with $\deg^+(v) \leq \gamma\abs{B}$ for some sufficiently small constant $\gamma > 0$. For each $v\in C$, we then have $\deg^-(v) \geq n - \gamma\abs{B}$, so $v$ is adjacent to most vertices of $B$. If the bipartite graph $G[A,C]$ had density $\Omega(1)$, then an application of \cref{lem:structure} would yield a large subset $A' \subseteq A$ whose degrees are concentrated into a small subset $B' \subseteq B$, contradicting the termination of the above procedure. Hence, $G[A, C]$ must be sparse. Since every vertex $v \in A$ has at least $n - r$ neighbours in $V_3 \supseteq C$, this forces $\abs{C} = \cO(r)$. 

Because $C$ is small, every $2tr$-booster in $H[A,B]$ has a large common neighbourhood in $V_3 \setminus C$. Now consider a triangle $abc \in A \times B \times (V_3 \setminus C)$. Recall that $\deg^+(a) \leq r$ and $\deg^+(c) \geq \gamma \abs{B}$. We are able to show that $\gamma \abs{B}$ is significantly larger than $r$, and so at least one of the edges $ab$ or $bc$ is in fact a $\gamma \abs{B}/4$-booster. As there are many such triangles, there must be many edges of this type. If a substantial fraction of them lie in $H[A, B]$, a counting argument produces a $K_{t,t,t}$. Consequently, most of these large boosters must lie in $G[B,V_3 \setminus C]$. Finally, this allows us to extract, from $B$, an $r'$-squad with $r' \approx \gamma\abs{B}/(3t) > r$, contradicting the maximality of $r$.

\begin{proof}[Proof of \cref{lem:lotsheavywin}]
    We begin by choosing an $r$-squad $A$ in $G$ for which $r$ is maximal. \Cref{lem:squad} guarantees that such a $A$ exists with $r \geq 100k$. By relabelling we may assume $A \subset V_1$. Let $H$ be the graph of forward $2tr$-booster edges incident with vertices in $A$ and let $B$ be the end-points of the edges in $H$ in $V_2$. Note that $\abs{A} \geq k/800 \geq K/80000 \cdot n^{1 - 1/t}$ and $e_H(a, B) \geq k/800 \geq K/80000 \cdot n^{1 - 1/t}$ for all $a \in A$.

    Let $K_{\ref{lem:structure}}$ be the constant obtained by applying \cref{lem:structure} to $t$ and $\lambda = 1/2$. We will now iteratively shrink the pair $(A, B)$ to a subgraph with higher density of $2tr$-boosters. We let $m \geq 0$ be maximal such that there are sets $A_m \subset \dotsb \subset A_1 \subset A_0 \coloneqq A$ and $B_m \subset \dotsb \subset B_1 \subset B_0 \coloneqq B$ such that, for all $i \in [m]$,
    \begin{itemize}
        \item $\abs{A_i} \geq \abs{A_{i - 1}}/2^{3t}$,
        \item $\abs{B_i} \leq \abs{B_{i - 1}}/2^{6t}$, and
        \item $e_H(a, B_i) \geq e_H(a, B_{i - 1}) - K_{\ref{lem:structure}} \cdot \abs{B_{i - 1}}^{1 - 1/t}$ for each $a \in A_i$. 
    \end{itemize}
    By the maximality of $m$, $G$ contains no subpair $A_{m + 1} \subset A_m$ and $B_{m + 1} \subset B_m$ satisfying all of the above. Now
    \begin{equation}\label{eq:sizeofAm}
        \abs{A_m} \geq \frac{\abs{A}}{2^{3tm}} \ge \frac{K}{80000} \cdot\frac{n^{1 - 1/t}}{2^{3tm}} \geq \frac{K}{80000} \Bigl(\frac{n}{2^{6tm}}\Bigr)^{1 - 1/t} \geq \frac{K}{80000} \cdot \abs{B_m}^{1 - 1/t}, 
    \end{equation}
    where the second equality uses $2 \leq 2^{2(1 - 1/t)}$ and the final inequality uses $\abs{B_0} \leq n$. For each $a \in A_m$ we have
    \begin{equation}\label{eq:degreetoBm}
        \begin{aligned}
            e_H(a, B_m) & \geq e_H(a, B) - K_{\ref{lem:structure}} \sum_{i = 0}^{m-1} \abs{B_i}^{1 - 1/t} \geq \frac{k}{800} - K_{\ref{lem:structure}} \cdot \abs{B}^{1 - 1/t} \sum_{i = 0}^{m-1} 2^{-3ti} \\ 
            & \geq \frac{k}{800} - K_{\ref{lem:structure}} \cdot n^{1 - 1/t} \cdot \frac{1}{1 - 2^{-3t} } \geq \frac{k}{1000},
        \end{aligned}
    \end{equation}
    where the second inequality uses $1/2 \leq 1 - 1/t$, the third is obtained by summing the geometric series, and the last one follows since $k \geq K n^{1 - 1/t}/100$ and $K$ is large relative to $K_{\ref{lem:structure}}$. Note that this implies
    \begin{equation}\label{eq:sizeofBm}
        \abs{B_m} \geq \frac{k}{1000}.
    \end{equation}
    We now improve this bound. Suppose that $\abs{B_m} \leq 6t^2 2^{6t} r$. Then, since $e_H(a, B_m) \geq k/1000$ for all $a \in A_m$, the density $\lambda$ of $H[A_m, B_m]$ is at least $k/(1000\abs{B_m})$. Note, also using \eqref{eq:sizeofAm}, that 
    \begin{equation*}
        \lambda \abs{A_m} \geq  \frac{K}{10^8} \biggl(\frac{k}{\abs{B_m}} \cdot \abs{B_m}^{1 - 1/t}\biggr) \geq \frac{K^2}{10^{10}} \cdot \frac{n^{1 - 1/t}}{\abs{B_m}^{1/t}} \geq \frac{K^2}{10^{10}} \cdot n^{1 - 2/t} \geq t,
    \end{equation*}
    provided $K$ is large enough. Thus we may apply \cref{lem:coneigh_density} to obtain vertices $a_1, \dotsc, a_t \allowbreak \in A_m$ and $B^\ast \subset N(a_1, \dotsc, a_t) \cap B_m$ such that
    \begin{equation*}
        \abs{B^\ast} \geq \frac{\lambda^t \abs{B_m} }{2 e^t} \stackrel{\eqref{eq:sizeofBm}}{\geq} \frac{k^t}{10^{4t} \cdot \abs{B_m}^{t - 1}} \geq \frac{K^t}{10^{9t^2}} \cdot  n^{t - 1} r^{1 - t} \geq Kn^{t - 1} r^{1 - t},
    \end{equation*}
    where the third inequality uses $\abs{B_m} \leq 6t^2 2^{6t} r \leq 10^{3t} r$ and the final inequality uses that $K$ is sufficiently large in terms of $t$. Since the edges of $H$ are $2tr$-heavy, every vertex of $B^\ast$ has at least $2tr$ neighbours in $V_3$. Furthermore, $\deg^-(a_i) = \deg(a_i) - \deg^+(a_i) \geq n - r$ and so each $a_i$ has at most $r$ non-neighbours in $V_3$. Thus each vertex of $B^\ast$ has at least $2tr - tr = tr$ neighbours in $N(a_1, \dotsc, a_t) \cap V_3$. Thus the number of edges between $B^\ast$ and $N(a_1, \dotsc, a_t) \cap V_3$ is at least
    \begin{equation*}
        \abs{B^\ast} \cdot tr \geq \abs{B^\ast} \cdot k + Kn^{t-1}r^{2-t}\geq  \frac{K}{100} \cdot(\abs{B^\ast} \cdot n^{1 - 1/t} + n),
    \end{equation*}
    where the final inequality uses $r \leq n$. Hence, provided $K$ is large enough, \cref{thm:KST} guarantees there is a $K_{t, t}$ in $G[B^\ast, N(a_1, \dotsc, a_t) \cap V_3]$. Together with $a_1, \dotsc, a_t$ this gives a $K_{t, t, t}$ in $G$. Hence, we may and will assume that
    \begin{equation}\label{eq:truesizeofBm}
        \abs{B_m} \geq 6t^2 2^{6t} r. 
    \end{equation}
    Let
    \begin{equation*}
        C \coloneqq \set{c \in V_3 \colon \deg^+(c) \leq \abs{B_m}/(t2^{6t})}.
    \end{equation*}
    Suppose that $\abs{C} \geq 2r$. Every vertex $a \in A_m$ satisfies $\deg^+(a) = \deg(a) - \deg^-(a) \geq n - r$. Thus, every vertex of $A_m$ has at most $r$ non-neighbours in $V_3$ and so is adjacent to at least half of $C$. Thus $G[C, A_m]$ has density $\lambda \geq 1/2$ and $\lambda \abs{C} \geq r \geq t$. Applying the structural lemma, \cref{lem:structure}, to $G[C, A_m]$ with third part $B_m$ gives vertex sets $A_{m + 1} \subset A_m$ and $B_{m + 1} \subset B_m$ such that
    \begin{itemize}[noitemsep]
        \item $\abs{B_{m + 1}} \leq t \cdot \abs{B_m}/(t2^{6t}) = \abs{B_m}/2^{6t}$,
        \item $\abs{A_{m + 1}} \geq \frac{1}{2}(1/2e)^t \abs{A_m} - \abs{B_m}^{1/t}$,
        \item for all $a \in A_{m + 1}$: $e(a, B_{m + 1}) \geq e(a, B_m) - K_{\ref{lem:structure}} \cdot \abs{B_m}^{1 - 1/t}$.
    \end{itemize}
    The first of these properties follows from the definition of $C$. From inequality~\eqref{eq:sizeofAm}, provided $K$ is large enough, $\abs{B_m}^{1/t} \leq \abs{B_m}^{1 - 1/t} \leq \frac{1}{4}(1/2e)^t \abs{A_m}$ and so $\abs{A_{m + 1}} \geq \frac{1}{4}(1/2e)^t \abs{A_m} \geq \abs{A_m}/2^{6t}$. But then the subpair $(A_{m + 1}, B_{m + 1})$ contradicts the maximality of $m$. Hence we may and will assume that
    \begin{equation}\label{eq:truesizeofC}
        \abs{C} \leq 2r.
    \end{equation}
    Partition $B_m$ into $B'_m$ and $B''_m$ where
    \begin{align*}
        B'_m & \coloneqq \set{b \in B_m \colon \deg^+(b) \leq \abs{B_m}/(3t^2 2^{6t})}, \\
        B''_m & \coloneqq \set{b \in B_m \colon \deg^+(b) > \abs{B_m}/(3t^2 2^{6t})}.
    \end{align*}
    
    \paragraph{Case 1: \texorpdfstring{$e_H(A_m, B'_m) \geq \frac{1}{2} e_H(A_m, B_m)$}{Dense to low forward degree}.} Then
    \begin{equation}\label{eq:AmBmprimeedges}
        e_H(A_m, B'_m) \geq \tfrac{1}{2} e_H(A_m, B_m) \stackrel{\eqref{eq:degreetoBm}}{\geq} \tfrac{1}{2} \abs{A_m} \cdot \tfrac{k}{1000},
    \end{equation}
    and so $\abs{B'_m} \geq k/2000$.

    Since each edge of $H$ is a $2tr$-booster, every vertex of $B'_m$ has at least $2tr$ neighbours in $V_3$. By~\eqref{eq:truesizeofC}, at most $2r$ of these are in $C$. Thus each vertex of $B'_m$ has at least $2r \geq k/800$ neighbours in $V_3 \setminus C$. Consider an edge $bv$ with $b \in B'_m$ and $v \in V_3 \setminus C$. By the definitions of $B'_m$ and $C$, we have
    \begin{equation*}
        \deg^+(v) - \deg^+(b) \geq \frac{\abs{B_m}}{t 2^{6t}} - \frac{\abs{B_m}}{3 t^2 2^{6t}} \geq 2t \cdot \frac{\abs{B_m}}{3t^2 2^{6t}},
    \end{equation*}
    and so every edge between $B'_m$ and $V_3 \setminus C$ is a $2tr'$-booster for $r' \coloneqq \abs{B_m}/(3t^2 2^{6t})$ and therefore every vertex of $B'_m$ is incident to at least $k/800$ forward edges that are $2tr'$-boosters.
    Also, every vertex of $B'_m$ has forward degree at most $r'$, by definition. Inequality~\eqref{eq:truesizeofBm} implies that $r' \geq 2r$ and so $B'_m$ forms an $r'$-squad contradicting the maximality of $r$ unless $\abs{B'_m} \leq k/800$. 
    
    If $\abs{B'_m} \leq k/800$, then \eqref{eq:AmBmprimeedges} implies the density of $H[A_m, B'_m]$ is at least $2/5$. Every edge of $H[A_m, B'_m]$ is $2tr$-heavy and $\deg^+(a) \leq r$ for all $a \in A_m \subset A$. Finally, by \eqref{eq:sizeofBm} and \eqref{eq:sizeofAm}, $\abs{A_m} \geq K/80000$ and $\abs{B_m} \geq K/10^5 \cdot n^{1 - 1/t}$ and so, provided $K$ is large enough, \cref{lem:denseheavywin} implies that $G$ contains $K_{t, t, t}$.

    \paragraph{Case 2: \texorpdfstring{$e_H(A_m, B''_m) \geq \frac{1}{2} e_H(A_m, B_m)$}{Dense to high forward degree}.} Then, by \eqref{eq:degreetoBm}, we have $e_H(A_m, B''_m) \geq \frac{1}{2} \cdot \abs{A_m} \cdot k/1000$ and so the density $\lambda$ of $H[A_m, B''_m]$ is at least $k/(2000\abs{B''_m})$. Note that 
    \begin{equation*}
        \lambda \abs{A_m} \stackrel{\eqref{eq:sizeofAm}}{\geq} \frac{K}{10^9} \cdot \frac{k}{\abs{B''_m}} \cdot \abs{B_m}^{1-1/t} \geq \frac{K^2}{10^{11}} \cdot \frac{n^{1-1/t}}{\abs{B_m}^{1/t}} \geq t,
    \end{equation*} 
    and so we may apply \cref{lem:coneigh_density} to obtain vertices $a_1, \dotsc, a_t \in A_m$ and and a set of vertices $B^\ast \subset N(a_1, \dotsc, a_t) \allowbreak \cap B''_m$ such that
    \begin{equation}\label{eq:sizeofBstar}
        \abs{B^\ast} \geq \tfrac{1}{2} (\lambda/e)^t \abs{B''_m} \geq \frac{k^t}{10^{5t} \abs{B''_m}^{t-1}} \geq \frac{K^t}{10^{7t}} n^{t - 1} \abs{B''_m}^{1 - t} \geq K n^{t - 1} \abs{B''_m}^{1 - t}.
    \end{equation}
    By the definition of $B''_m$, every vertex of $B^\ast \subset B''_m$ has at least $\abs{B_m}/(3t^2 2^{6t})$ neighbours in $V_3$. Also, since the edges of $H$ are $2tr$-heavy, every vertex of $B^\ast$ has at least $2tr$ neighbours in $V_3$. Furthermore, $\deg^-(a_i) = \deg(a_i) - \deg^+(a_i) \geq n - r$ and so each $a_i$ has at most $r$ non-neighbours in $V_3$. Thus each vertex of $B^\ast$ has at least
    \begin{equation*}
        \max\set{\abs{B_m}/(3t^2 2^{6t}), 2tr} - tr \geq \tfrac{1}{2}\max\set{\abs{B_m}/(3t^2 2^{6t}), 2tr} \geq \abs{B_m}/(6t^2 2^{6t})
    \end{equation*}
    neighbours in $N(a_1, \dotsc, a_t)$. Thus the number of edges between $B^\ast$ and $N(a_1, \dotsc, a_t) \cap V_3$ is at least
    \begin{align*}
        \abs{B^\ast} \cdot \abs{B_m}/(6t^2 2^{6t}) & \stackrel{\eqref{eq:sizeofBm}}{\geq} \frac{\abs{B^\ast} \cdot k}{2000 \cdot 6t^2 2^{6t}} + \frac{\abs{B^\ast} \cdot \abs{B_m}}{6t^2 2^{6t+1}} \\
        & \stackrel{\eqref{eq:sizeofBstar}}{\geq} \frac{K}{10^{10t}} \Bigl(\abs{B^\ast} \cdot n^{1 - 1/t} + n^{t - 1} \cdot \abs{B''_m}^{2 - t}\Bigr) \\
        & \geq \frac{K}{10^{10t}} (\abs{B^\ast} \cdot n^{1 - 1/t} + n),
    \end{align*}
    where the second inequality used $200000 \cdot 6t^22^{6t} \leq 10^{10t}$ and the final inequality used $\abs{B''_m} \leq n$. Thus, provided $K$ is large enough, \cref{thm:KST} guarantees that there is a $K_{t, t}$ in $G[B^\ast, N(a_1, \dotsc, a_t) \cap V_3]$. Together with $a_1, \dotsc, a_t$ this gives a $K_{t, t, t}$ in $G$.
\end{proof}

We are now ready to conclude the proof of our main result, \cref{thm:main}. Let $G$ be an $n \times n \times n$ tripartite graph with minimum degree $n + Kn^{1 - 1/t}$. \Cref{lem:initialconfig} guarantees there is an integer $k \geq K/100 \cdot n^{1 - 1/t}$ such that $G$ contains either a bipartite subgraph of booster edges with parts of size at least $k/200$ and density at least $1/200$ or $k/12$ vertices in one part each having forward degree at most $100k$ and each being incident to at least $k/200$ forward edges that are $k/100$-heavy. If the former occurs, then \cref{lem:denseboosterwin} implies that, provided $K$ is large enough, $G$ contains $K_{t, t, t}$. If the latter occurs, then \cref{lem:lotsheavywin} implies that, provided $K$ is large enough, $G$ contains $K_{t, t, t}$. In either case $G$ contains $K_{t, t, t}$, as required.

\section{Concluding remarks}\label{sec:conclud}

While \cref{thm:main} provides a satisfactory answer to Bollob\'{a}s, Erd\H{o}s, and Szemer\'{e}di's problem, \cref{prob:BES}, one could ask an even more general question.

\begin{prob}\label{prob:general}
    For each tripartite graph $H$, what is the smallest $\tau = \tau(n)$ such than any $n \times n \times n$ tripartite graph with minimum degree at least $n + \tau$ contains $H$\textup{?}
\end{prob}

For arbitrary $H$ this may be very hard indeed given our limited understanding of bipartite Tur\'{a}n numbers. One case that may be more approachable is $H = K_{t, s, r}$ for $t \leq s \leq r$.

\begin{conj}
    For all positive integers $t \leq s \leq r$, there is a constant $K$ such that, if $G$ is an $n \times n \times n$ tripartite graph and $\delta(G) \geq n + K n^{1 - 1/t}$, then $G$ contains $K_{t, s, r}$.
\end{conj}

If true, this would be best possible up to the constant $K$, owing to the extremal examples of \cref{sec:extremal}. Much of our argument still works here (note that the K\H{o}v\'{a}ri-S\'{o}s-Tur\'{a}n theorem, \cref{thm:KST}, actually gives a $K_{t, s}$ provided $K$ is large enough), but when $s$ is much bigger than $t$ case (ii) of \cref{lem:initialconfig} does not provide a high enough density of heavy edges to proceed.

Another possible direction would be to consider an $r$-partite version for $r \geq 4$. Bollob\'{a}s, Erd\H{o}s, and Szemer\'{e}di~\cite{BES75} asked what minimum degree $c_r n + 1$ guarantees the presence of $K_r$ in an $n \times \dotsb \times n$ $r$-partite graph -- that is, in an $r$-partite graph where all parts have size $n$. This spurred on a line of research~\cite{Alon88,Jin92} with the solution for even $r$ given by Szab\'{o} and Tardos~\cite{SzaboTardos06} and for odd $r$ by Haxell and Szab\'{o}~\cite{HaxellSzabo06}. It is natural to ask the following.

\begin{prob}
    For integers $r \geq 4$ and $t \geq 2$, what is the smallest $\tau = \tau(n)$ such that every $n \times \dotsb \times n$ $r$-partite graph with minimum degree at least $c_r n + \tau$ contains $K_r(t)$\textup{?}
\end{prob}

\paragraph{Acknowledgements.} We would like to thank Peter Allen, Julia B\"{o}ttcher, Jasmin Katz, and Jozef Skokan for valuable discussions. We thank the anonymous referees for their thorough reading and helpful suggestions.

{
\fontsize{11pt}{12pt}
\selectfont
	
\hypersetup{linkcolor={red!70!black}}
\setlength{\parskip}{2pt plus 0.3ex minus 0.3ex}

\newcommand{\etalchar}[1]{$^{#1}$}

}

\end{document}